\documentclass[11pt]{amsart}
\usepackage{amsmath}
\usepackage{amsfonts}
\usepackage{amssymb}

\usepackage[T1]{fontenc}
\usepackage[ansinew]{inputenc}
\usepackage{amssymb,amsmath,amsthm,eucal,hyperref,mathrsfs,array,graphicx,xcolor,bbm,enumerate}
\usepackage[all]{xy}
\usepackage{fullpage}

\newtheorem{thm}{Theorem}
\newtheorem{lemma}[thm]{Lemma}
\newtheorem{rem}[thm]{Remark}
\newtheorem{defn}[thm]{Definition}
\newtheorem{prop}[thm]{Proposition}
\newtheorem{cor}[thm]{Corollary}
\newtheorem{claim}[thm]{Claim}
\newcommand{\Q}{\mathbb Q}
\newcommand{\Z}{\mathbb Z}
\newcommand{\R}{\mathbb R}
\renewcommand{\H}{\mathbb H}
\newcommand{\HH}{\mathbb H}
\newcommand{\N}{\mathbb N}
\newcommand{\D}{\mathbb D}
\newcommand{\E}{\mathbb E}
\newcommand{\I}[1]{\mathbf{1}_{\left \{#1\right \}}}
\renewcommand{\P}{\mathbb P}
\renewcommand{\1}{\mathbf 1}
\newcommand{\eps}{\epsilon}
\newcommand{\crad}{\text{crad}}
\newcommand{\sgn}{\text{sign}}
\newcommand{\CLE}{\text{CLE}}

\renewcommand{\d}{{d}}

\newcommand{\revise}[1]{{#1}}

\begin{document}
\title{On bounded-type thin local sets \\ of the two-dimensional Gaussian free field}
\author{Juhan Aru, Avelio Sep\'ulveda, Wendelin Werner}

\address{Department of Mathematics, ETH Z\"urich, R\"amistr. 101, 8092 Z\"urich, Switzerland} 

\email{juhan.aru@math.ethz.ch}
\email{leonardo.sepulveda@math.ethz.ch}
\email{wendelin.werner@math.ethz.ch}

\begin {abstract}
We study  certain classes of local sets of the two-dimensional Gaussian
free field (GFF) in a simply-connected domain, and their relation to
the conformal loop ensemble CLE$_4$ and its
variants. More specifically, we consider bounded-type thin local sets (BTLS), where thin means that the local set is small in size, and bounded-type means that the harmonic function describing the mean value of the field away from the local set is
bounded by some deterministic constant.
We show that a local set is a BTLS if and only if it is contained in some nested version of the CLE$_4$ carpet, and prove that all BTLS are necessarily connected to the boundary of the domain.
We also construct all possible BTLS for which the corresponding harmonic function
takes only two prescribed values and
show that all these sets (and this includes the case of CLE$_4$) are in fact measurable functions of the GFF. 
\end {abstract}
\maketitle 

\section {Introduction}

\subsection{General introduction}

Thanks to the recent works of Schramm, Sheffield, Dub\'edat, Miller and others (see \cite {SchSh,SchSh2,Dub,ShQZ,MS1,MS2,MS3,MS4} and the references therein), it is known
that many structures built using Schramm's SLE curves  can be naturally coupled with the planar Gaussian Free Field (GFF). For instance, even though the GFF is not a continuous function, SLE$_4$  and the conformal loop ensemble CLE$_4$ can be viewed as  ``level lines'' of the GFF.
In these couplings, the notion of {\em local sets of the GFF} and their properties turn out to be instrumental.
This general abstract concept appears already in the study of Markov random field in the 70s and 80s (see in particular \cite{Roz}) and can be viewed as the natural generalization of 
stopping times for multidimensional time.  
More precisely, if the random generalized function $\Gamma$ is a (zero-boundary) GFF in a  planar domain $D$, a random set $A$  is said to be a local set for $\Gamma$
if  the conditional distribution of the GFF in $D \setminus A$ given $A$ has the law of the sum of a (conditionally) independent GFF $\Gamma^A$ in $D \setminus A$ with
a random harmonic function $h_A$ defined in $D \setminus A$. This harmonic function can be interpreted as the harmonic extension to $D \setminus A$ of the values of the GFF on $\partial A$.

In the present paper, we are interested in a special type of local sets 
that we call {\em bounded-type thin local sets} (or BTLS -- i.e., Beatles). Our definition of BTLS imposes three type of conditions, on top of being local sets: 
\begin {enumerate} 
\item There exists a constant $K$ such that almost surely, $| h_A | \le  K$ in the complement of $A$. 
\item The set $A$ is a thin local set as defined in \cite{Se,WWln2}, which implies in particular that the harmonic function $h_A$ carries itself all the information about the
GFF also on $A$. More precisely, this means here (because we also have the first condition) that for any smooth test function $f$, the random variable $(\Gamma,f)$ is almost surely equal to  $(\int_{D \setminus A }  h_A(x) f(x) dx) + (\Gamma^A, f)$. 
\item Almost surely, each connected component of $A$ that does not intersect $\partial D$ has a neighbourhood that does intersect no other connected component of $A$. 
\end {enumerate}
If the set $A$ is a BTLS with constant $K$ in the first condition, we say that it is a $K$-BTLS.

The first two conditions are the key ones, and they appear at first glance antinomic (the first one tends to require the set $A$ not to be too small, while the second requires it to be 
small), so that it is not obvious that such BTLS exist at all. Let us stress that Condition (1) is highly non-trivial: The GFF is only a generalized function and we loosely speaking require the GFF here to be bounded by a constant on $\partial A$. Notice for instance that a deterministic non-polar set is not a BTLS because the corresponding harmonic function is not bounded, so that a non-empty BTLS is necessarily random. 
As we shall see, the second condition which can be interpreted as a condition on the size of the set $A$ (i.e., it can not be too large), can be compared in the previous analogy between local sets and stopping times to requiring stopping times to be uniformly integrable. 

The third condition implies in particular that $A$ has only countably many connected components that do not intersect $\partial D$. This  is somewhat restrictive as there exist totally disconnected sets that are non-polar for the Brownian motion (for instance a Cantor set with Hausdorff dimension in $(0,1)$), and that can therefore have a non-negligible boundary effect for the GFF. We however believe that this third condition is not essential and could be disposed of (i.e., all statements would still hold without this condition), and we comment on this at the end of the paper. 

We remark that the precise choice of definitions is not that important here. We will see, in fact, that the combination of these three conditions implies much stronger statements. For instance, the (upper) Minkowski dimension of such a $K$-BTLS $A$ is necessarily bounded by some constant $d(K)$ that is smaller than 2 (and this is stronger than the second condition), and that $\partial D \cup A$ has almost surely only one connected component (which implies the third condition). 
 
Non-trivial BTLS do exist, and the first examples are provided by SLE$_4$, CLE$_4$ and their variants (for their first natural coupling with the GFF) as shown in \cite {Dub,MS,SchSh2}.
In the present paper, among other things, we prove that any BTLS $A$ is contained in a nested version of CLE$_4$ and that $A \cup \partial D$ is necessarily connected. From our proofs it also follows that the CLE$_4$ and its various generalizations are deterministic functions of the GFF.
Keeping the stopping time analogy in mind, one can compare these results with the following feature of one-dimensional Brownian motion $B$ started from the origin: 
if $T$ is a stopping time with respect to the filtration of $B$ such $|B_T| \le K$ almost surely, then either $T \le \inf \{ t  \ge  0  :  |B_t| = K \}$ almost surely, or $T$ is not integrable.
In other words, the CLE$_4$ and its nested versions and variants are the field analogues of one-dimensional exit times of intervals.  

One important general property of local sets, shown in \cite {SchSh2} and used extensively in \cite {MS1}, is that when $A$ and $B$ are two local sets coupled with the same GFF, in such a way 
that $A$ and $B$ are conditionally independent given 
the GFF, then their union $A \cup B$ is also a local set. It seems quite natural to expect that it should be possible to describe the harmonic function $h_{A \cup B}$ simply in terms of $h_A$
and $h_B$, but deriving a general result in this direction appears to be, somewhat surprisingly, tricky. 
The present approach provides a way to obtain results in this direction in the case of BTLS: we shall see that the union of two bounded-type thin local sets is always a BTLS.
 
In further work, we plan to investigate some features of our approach to the imaginary geometry setting,
and hope that this might be of some help to enlighten some of the technical issues dealt with in \cite {MS1,MS2,MS3}.

\subsection{An overview of results}

We now state more precisely some of the results that we shall derive. Throughout the present section, $D$ will denote a simply connected domain with non-empty boundary (so that $D$ is conformally equivalent to the unit disk). 

Let us first briefly recall some features of the coupling between CLE$_4$ and GFF in such a domain $D$. 
Using a branching-tree variant of SLE$_4$ introduced by Sheffield in \cite{She}, it is possible to define a certain random conformally invariant family of marked open sets $(O_j, \eps_j)$, where the $O_j$'s form 
a disjoint family of open subsets of the upper half-plane, and the marks $\eps_j$ belong to $\{-1, 1\}$. The complement $A$ of $\cup_j O_j$ is called a CLE$_4$ carpet (see Figure \ref {figcle4}) and 
its Minkowski dimension is in fact almost surely  equal to $15/8$ (see \cite{NW,SchShW}). As pointed out 
by Miller and Sheffield \cite {MS}, the set $A$ can be coupled with the GFF as a BTLS in a way that $h_A$ is constant and equal to $2\lambda \eps_j$ (with $\lambda=\sqrt{\pi/8}$) in each $O_j$.
\begin{figure}[ht!]
\begin{center}
\includegraphics[width=3.1in]{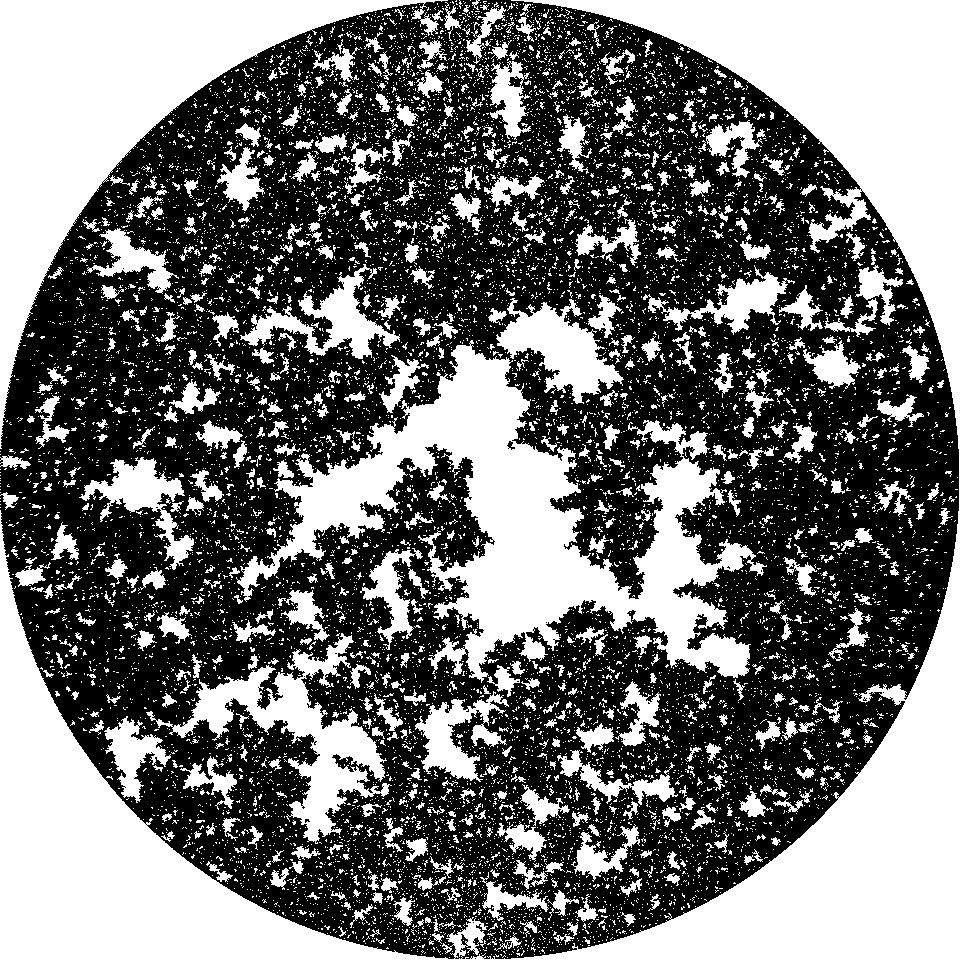}
\begin{flushright}
	\footnotesize  Simulation by David B. Wilson.
\end{flushright}
\caption{\label{figcle4}  The set in black represents the CLE$_4$ carpet and the white holes are the open sets $O_j$. }
\end{center}
\end{figure} 

By the property of local sets, conditionally on $A$,  the field $\Gamma - h_A$ consists of independent Gaussian free fields  $\Gamma^j$ inside each $O_j$. 
We can then iterate the same construction independently for each of these GFFs $\Gamma^j$, using a new CLE$_4$ in each $O_j$. 
In this way, for each given $z \in \HH$, if we define $O^1 (z)$ to be the open set $O_j$ that contains $z$ (for each fixed $z$, this set almost surely exists) and set $\eps^1 (z) = \eps_j$,
we then get a second set $O^2 (z) \subset O^1 (z)$ corresponding to the nested CLE$_4$ in $O^1 (z)$ and a new mark $\eps^2 (z)$. Iterating the procedure, we  
obtain for each given $z$, an almost surely decreasing sequence of open sets $O^n (z)$ and a sequence of marks $\eps^n (z)$ in $\{ -1 , + 1\}$. When $z' \in O^n (z)$, then $O^n (z')= O^n (z)$ and  
$\eps^n (z') = \eps^n (z) $, so that $\eps^n (\cdot)$ can be viewed as a constant function in $O^n (z)$. 
Furthermore, for each $z \in D$, the sequence $\Upsilon_n (z)= \sum_{m=1}^n \eps^m (z)$ is a simple random walk. 

For each $n$, the complement $A_n$ of the union of all $O^n (z)$ with $z \in \Q^2 \cap D$ 
is then a BTLS and the corresponding harmonic function $h_n$ is just $2 \lambda \Upsilon_n ( \cdot)$. The set $A_n$ is called the nested CLE$_4$ of level $n$ 
(we refer to it as CLE$_{4,n}$ in the sequel). It is in fact easy to see 
that one can recover the GFF $\Gamma$ from the knowledge of all these pairs $(A_n, h_n)$ for $n \ge 1$ (because for each smooth test function $f$, the sequence $(\Gamma, f) - (h_n, f)$  converges to $0$ in $L^2$). 

For each integer $M \ge 1$, one can then define for each $z$ in $\HH$, the random variable $M(z)$ to be the smallest $m$ at which $| \Upsilon_m (z) | = M$. As for each fixed $z$, the sequence $(\Upsilon_n (z))$
is just a simple random walk, $M(z)$ is almost surely finite. 
The complement $A^M$ of the union of all these $O^{M(z)} (z)$ for $z$ with rational coordinates defines a $2\lambda M$-BTLS: 
Indeed, the corresponding harmonic function is then constant in each connected component $O^{M(z)}(z)$ and takes values in  $\{ 2M \lambda, -2M \lambda \}$, and 
we explain in Section \ref{CLE4} that the Minkowski dimension of $A^M$ is almost surely bounded by  $d=2 - (1/ (8M^2)) < 2$ which imply the thinness. 
We refer to this set $A^M$ as a CLE$_4^M$ (mind that $M$ is in superscript, as opposed to CLE$_{4,m}$, where we just iterated $m$ times the CLE$_4$). 

A special case of our results is that this $\CLE_4^M$ is the only BTLS with harmonic function taking its values in $\{ -2M \lambda, 2M \lambda\}$. More precisely: 
\begin{prop}\label{cledesc}
Suppose that a CLE$_4^M$ that we denote by $C$ is coupled with a GFF in the way that we have just described. 
Suppose that $A$ is a BTLS coupled to the same GFF $\Gamma$, such that the corresponding harmonic function $h_A$ takes also its values in $\{ - 2 M \lambda, 2M \lambda \}$.
Then, $A$ is almost surely  equal to $C$. 
\end {prop}
In particular, for $M=1$ and taking $A$ to be another CLE$_4$ coupled with the same GFF, this implies that any two $\CLE_4$ that are coupled with the same GFF as local sets with harmonic function in $\{ -2 \lambda, 2 \lambda \}$ are almost surely identical. 
This BTLS approach therefore provides a rather short proof of the fact that {\em the first layer CLE$_4$ and therefore also all nested CLE$_{4,m}$'s are deterministic functions of the GFF} (hence, that the information provided 
by the collection of nested labelled CLE$_{4,m}$'s is equivalent to the information provided by the GFF itself). 
This fact is not new and is due to Miller and Sheffield  \cite {MS}, who have outlined the proof in private discussions, presentations and talks (and a paper in preparation). 
Our proof follows a somewhat different route than the one proposed by Miller and Sheffield,  although the basic ingredients are similar (absolute continuity properties of the GFF, basic properties of local sets and the fact that SLE$_4$ itself is a deterministic function of the GFF).

Let us stress that the condition of being thin is important.
Indeed, if we consider the union of a CLE$_4$ with, say, the component of its complement that
contains the origin, we still obtain a local set for which the corresponding harmonic function is in  $\{ - 2  \lambda, 2 \lambda \}$, yet it is almost surely clearly not contained in any of the $\CLE_4^M$
for $M\ge 1$.

We will also characterize all possible BTLS such that the harmonic function can take only two possible prescribed values, and study their properties. We will in particular derive the following facts: 
\begin {prop}
 \label {cledesc2}
Let us consider $-a < 0 < b$. 
\begin {enumerate}
\item
When $a+b  < 2 \lambda$, it is not possible to construct a BTLS $A$ such that $h_A \in \{ -a , b \}$ almost surely.
\item
When $a+b \ge 2 \lambda$, it is possible to construct a BTLS $A_{-a,b}$ coupled with a GFF $\Gamma$ in such a 
way that $h_A \in \{ -a, b \}$ almost surely. Moreover, the sets $A_{-a,b}$ are
\begin{itemize} 
	\item Unique in the sense that if $A'$ is another BTLS coupled with the same $\Gamma$,  such that for all $z \in D$, $h_{A'} (z) \in \{ -a ,b \}$ almost surely, 
	then $A' = A_{-a, b}$ almost surely.  
	\item Measurable functions of the GFF $\Gamma$ that they are coupled with.
	\item Monotonic in the following sense: if $[a,b] \subset [a', b']$ and $-a < 0 < b$ with $b+a \ge 2\lambda$,  then almost surely, $A_{-a,b} \subset A_{-a', b'}$. 
\end{itemize}  
\end {enumerate}
\end {prop}
More information about the sets $A_{-a,b}$ and their properties (detailed construction, dimensions) as well as some generalizations are discussed in Section \ref {S6}. \revise{In particular, we present a new construction of CLE$_4$ only using SLE$_4(-1;-1)$ and SLE$_4(-1)$ processes.} 
Further properties of these sets will be described in \cite {AS}. 
The sets $A_{-a, b}$ are also instrumental in  \cite {ALS,QW}. 

Another type of result that we derive using similar ideas as in the proof of Proposition \ref {cledesc} goes as follows:
\begin{prop}\label{maxbtls}
If $B$ is a $(2M\lambda)$-BTLS associated to the GFF $\Gamma$, then $B$ is almost surely a subset of the $\CLE_4^{M+1}$  associated to $\Gamma$. 
\end{prop}
Notice that one would expect to conclude that $B \subseteq \CLE_4^{M}$ in Proposition \ref {maxbtls} (and this would mean that $\CLE_4^M$ is maximal among all $2\lambda M$-BTLS), but this seems to require some more technical work that we do not discuss in the present paper. 

{As a finite collection of BTLS is in fact a collection of $2\lambda M$-BTLS for some $M\in \N$, we see that  their union is almost surely contained in $\CLE_4^{M+1}$, and thus it is again a BTLS}. This type of facts helps to derive the following result, that we already mentioned earlier in this introduction:
\begin {prop}
\label {propnodisconnection}
If $A$ is a BTLS, then $A \cup \partial D$ is connected. 
\end {prop}

The structure of the paper is the following: We first recall some basic features about BTLS. Then, we discuss level lines of the GFF with non-constant boundary conditions and their boundary hitting behaviour. 
Thereafter, we recall features of the construction of the coupling of the GFF with $\CLE_4^M$. This sets the stage for the proofs of the propositions involving CLE$_4^M$. We then finally turn to 
Proposition \ref {cledesc2} and to the derivation of the dimensions of the sets $A_{-a,b}$.

\section{Local sets and BTLS} \label{TLS} \label {S3}
In this section, we quickly browse through basic definitions and properties of the GFF and of bounded-type local sets. 
We  only discuss items that are directly used in the current paper. For a more general discussion of local sets, 
thin local sets (not necessarily of bounded type), we refer to \cite {Se,WWln2}. 

\revise{Throughout this paper, the set $D$ denotes an open planar domain with a non-empty and non-polar boundary.} 
In fact, we will always at least assume that the complement of $D$  (a) has at most countably many connected components, (b) has only finitely many components that intersect each given compact subset of $D$,
(c) has no connected component that is a singleton; this last condition (c) excludes for instance sets like $\D \setminus K$, where $K \subset [0,1]$ is the middle Cantor set).
Recall that by a theorem of He and Schramm \cite{HS}, such domains $D$ are known to be conformally equivalent to circle domains (i.e. to $\D \setminus K$
\revise{or more conveniently for us to $\H \setminus K$}, where $K$ is a union of closed disjoint discs). 

	Recall that the (zero boundary) Gaussian Free Field (GFF) in a such a domain $D$  
	can be viewed as a centered Gaussian process $\Gamma = (( \Gamma, f))$ (we also sometimes write  $\Gamma^D$ when the domain needs to be specified) 
	indexed by the set of continuous functions $f$ with compact support in $D$, with covariance given by 
	$$ \E [(\Gamma,f) (\Gamma,g)]  =  \iint_{D\times D} f(x) G_D(x,y) g(y) \d x \d y $$ 
	where $G_D$ is the Green's function (with Dirichlet boundary conditions) in $D$, normalized such that $G_D(x,y)\sim (2\pi)^{-1} \log(1/|x-y|)$ as $x \to y$ for $y \in D$. 
	For this choice of normalization of $G$ (and therefore of the GFF), we set  $$\lambda=\sqrt{\pi/8}.$$ Sometimes, other normalizations are used in the literature: If $G_D (x,y) \sim c \log(1/|x-y|)$ as $x \to y$, then $\lambda$ should 
	be taken to be $(\pi/2)\times \sqrt {c}$. 
	Note that it is in fact possible and useful to define the random variable $(\Gamma, \mu)$ for any fixed Borel measure $\mu$, provided the energy 
	$\iint \mu (dx) \mu (dy) G_D (x,y)$ is finite.

	The covariance kernel of the GFF blows up on the diagonal, which makes it impossible to view $\Gamma$ as a random function.
	However, the GFF has a version that lives in some
	space of generalized functions acting on some deterministic space of smooth functions $f$ (see for example \cite{Dub}). This also justifies our notation $(\Gamma,f)$. 
	Let us stress that it is in general not possible to make sense of $(\Gamma,f)$ for {\em random} functions that are correlated with the GFF, even when $f=1_A$ is the
	indicator function of a random closed set $A$. Local sets form a class of random closed sets $A$, where this is (in a sense) possible. Here, by a random closed set we mean a random variable in the space of closed subsets of $\overline D$, endowed with the Hausdorff metric

\begin{defn}[Local sets]
	Consider a random triple $(\Gamma, A,\Gamma_A)$, where $\Gamma$ is a  GFF in $D$, $A$ is a \revise{random closed subset of $\overline D$} and $\Gamma_A$ is a random distribution that can be viewed as a harmonic function\revise{, $h_A$,} when restricted to compact subsets of 
	 $D \setminus A$.
	We say that $A$ is a local set for $\Gamma$ if conditionally on the couple $(A, \Gamma_A)$, the field $\Gamma - \Gamma_A$ is a  GFF in $D \setminus A$. 
\end {defn}

\revise{We use the different notation $h_A$ for the restriction of $\Gamma_A$ to $D\backslash A$, in order to emphasize that in $D\backslash A$ the generalized function $\Gamma_A$ is in fact a harmonic function and thus can be evaluated at any point $z \in D\backslash A$.}

When $A$ is a local set for $\Gamma$, we will define $\Gamma^A $ to be equal to $\Gamma - \Gamma_A$. Note that the conditional distribution of $\Gamma^A$ given $(A, \Gamma_A)$ is in fact
a function of $A$ alone. 

Notice that being a local set can also be seen as a property of the law of the couple $(\Gamma,A)$, as if one knows this law and the fact that $A$ is a local set of $\Gamma$, then 
one can recover $h_A$ as the limit when $n \to \infty$ of the conditional expectation of $\Gamma$ (outside of $A$) given $A$ and the values of $\Gamma$ in the smallest union of $2^{-n}$ dyadic squares that contains $A$. 
One can then recover $\Gamma^A$ which is equal $\Gamma - h_A$ outside of $A$, and finally 
one reconstructs $\Gamma_A = \Gamma - \Gamma^A$ (including on $A$). 
This argument shows in particular that any local set can be coupled in a unique way with a given GFF: if two random triples $(\Gamma, A,\Gamma_A)$ and $(\Gamma, A,\Gamma_A')$ are both local couplings, then $\Gamma_A$ and $\Gamma_A'$ are almost surely identical.

When $A$ is a local set for $\Gamma$, we will denote by ${\mathcal F}_A$ the $\sigma$-field generated by $(A, \Gamma_A)$. 
We will say that two local sets $A$ and $B$ that are coupled with the same Gaussian Free Field $\Gamma$ are {\em conditionally independent local sets of $\Gamma$} if the sigma-fields 
${\mathcal F}_A$ and ${\mathcal F}_{B}$ are conditionally independent given $\Gamma$. 

Let us list the following two properties of local sets (see for instance \cite {SchSh2} for derivations and further properties): 
\begin{lemma}\label{BPLS}    $\ $
\begin {enumerate}
\item When $A$ and $B$ are conditionally independent local sets for the GFF $\Gamma$, then $A \cup B$ is also a local set for $\Gamma$. 
\item When $A$ and $A'$ are  conditionally independent local sets for $\Gamma$ such that $A \subset A'$ almost surely, then for 
any smooth compactly supported test function $f$, $(\Gamma_{A}, f) = \E [ (\Gamma_{A'},f) | {\mathcal F}_A ]$ almost surely. 
\end {enumerate}
\end{lemma}

We now define local sets of bounded type: 
\begin{defn}[BTLS]\label{BTLSCND}
	Consider  a random relatively closed subset $A$ of $D$ (i.e. so that $D \setminus A$ is open) and $\Gamma$ a GFF defined on the same probability space.
Let $K\geq 0$, we say that $A$ is a $K$-BTLS if the following four conditions are satisfied: 
	\begin {enumerate}
	 \item $A$ is a local set of $\Gamma$.
	  \item Almost surely, $|h_A| \le K$ in $D \setminus A$.  
	 \item Almost surely, each connected component of $A$ that does not intersect $\partial D$ has a neighbourhood that does intersect no other connected component of $A$. 
	 \item $A$ is a thin local set in the sense defined in \cite{Se, WWln2}: for any smooth test function $f \in \mathcal{C}^\infty_0$, the random variable $(\Gamma_A,f)$ is almost surely equal to  $(\int_{D \setminus A }  h_A(x) f(x) dx)$.
	\end {enumerate}
If $A$ is a $K$-BTLS for some $K$, we say that it is a BTLS.
\end {defn}

Note also that it is in fact possible to remove all isolated points from a BTLS (of which there are at most countably many because of the third property)
without changing the property of being BTLS. Indeed, the bounded harmonic function $h_A$ can be extended to those points and a GFF does not see polar sets. We therefore  replace the third property by:
\begin {description}
\item [{\sl (3')}] Almost surely, $A$ contains no isolated points and each connected component of $A$ that does not intersect $\partial D$ has a neighbourhood that intersects no other connected component of $A$.
\end{description}
This reformulation is handy to keep our statements simple. \revise{In particular, this condition implies that for any $x\in A$ and for any neighbourhood $J$ of $x$, $J\cap A$ is not polar. To see this, it is enough to notice that $D\backslash A$ is conformally equivalent to a circle domain described in the beginning of this section.}

It is not hard to see that because the harmonic function is bounded, the condition (4) in the definition of BTLS
could be replaced (without changing the definition) by the fact that if we define $[A]_n$ to be the union $A_n$
of the $2^{n}$-dyadic squares that intersect $A$, then for any compactly supported smooth test function $f$ in $D$, 
the sequence of random variables $(\Gamma, f\1_{[A]_n})$ converges in probability to $0$. 
From a Borel-Cantelli argument, one can moreover see that this equivalent condition is implied by the stronger condition (see \cite{Se, WWln2}): 
\begin {description}
 \item [$(*)$] The expected volume of the $\eps$-neighbourhood of $A$ decays like $o(1 / \log (1 / \eps ))$ as $\eps \to 0$. 
\end {description}
In other words, if a set satisfies the first three conditions in our BTLS definition and $(*)$, then it is a BTLS.

Note that if $A$ and $B$ are two conditionally independent BTLS of $\Gamma$, then we know by Lemma \ref {BPLS} that $A \cup B$ is a local set, but not yet that it is a BTLS. 
In order to prove this, we will need to show it is thin and give an upper bound for the harmonic function $h_{A \cup B}$. 

It is not hard to derive the following related fact (that will be used later in the paper, at the end of the proof of the fact that the union of any two BTLS is a again a BTLS):
\revise{\begin{lemma}\label{local subset}
		Let $A$ and $A'$ are two  conditionally independent thin local sets of the same GFF $\Gamma$ such 
		that $A$ satisfies the condition (3') of Definition \ref{BTLSCND}, that $A'$ is a $K$-BTLS and that $A\subseteq A'$ almost surely. Then $A$ is a $K$-BTLS. 
\end{lemma}}
\revise{	
\begin {proof}
We need to check that $|h_A| \leq K$. Let us choose a smooth non-negative test function $\rho$ that is radially symmetric around the origin, of unit 
mass with support in the unit ball, and denote by $\rho_\eps^z$ the naturally shifted and scaled version of $\rho$, so that $\rho_\eps^z$ 
is radially symmetric around $z$, of unit mass and with support in the open ball $B(z,\eps)$. 
The final statement of Lemma \ref {BPLS} shows that when $\eps < d( z, \partial D)$, 
$\E [ (\Gamma_{A'},  \rho_\eps^z) | {\mathcal F}_{A} ] =  (\Gamma_{A},\rho_\eps^z)$ almost surely. As $A'$ is a $K$-BTLS, we know that almost surely, 
$ | ( \Gamma_{A'}, \rho_\eps^z) | \le K$, so that by Jensen's inequality, $| ( \Gamma_A, \rho_\eps^z ) | \le K$ almost surely. 
But as $\Gamma_A$ is equal to the harmonic function $h_A$ in the complement of $A$, we have $h_A(z) = \Gamma_A(\rho_\eps^z)$ as long as $d(z,A) > \eps$. 
Thus we conclude that with full probability on the event that $d(z, A) > \eps$, $|h_A (z) | \le K$. 
Since this holds almost surely for all $z$ with rational coordinate and every rational $\eps < d(z, \partial D)$ simultaneously, we conclude that almost surely $|h_A| \leq K$ in $D\setminus A$. 
\end{proof}}
 
The following two tailor-made lemmas will be used in the proof of the fact that any BTLS is contained in some CLE$_4^M$:
\revise{
\begin{lemma} \label{mes}
	Let $A$ and $B$ are two conditionally independent BTLS of the GFF $\Gamma$ such that $A$ almost surely satisfies condition $(*)$,
	and such that there exists  $k\in \R$ such that a.s. for all $z\notin A \cup B$, $|h_{A}(z)+k|\geq |h_{A \cup B}(z)+k| $.  Then $B\subseteq  A$ almost surely.
\end{lemma}}
\begin{proof}[Proof of Lemma \ref{mes}:] 
First, let us briefly explain why the conditions on $A$ and $B$ imply that $A \cup B$ is a thin local set. From Lemma \ref{BPLS}, we know that $A \cup B$ is local. To show that $A\cup B$ is thin, we write $(\Gamma,f\1_{[A\cup B]_n}) =  (\Gamma,f\1_{[B]_n}) + (\Gamma,f\1_{[A]_n\backslash [B]_n})$ and note that the first term converges to 0 thanks to the definition and the second one thanks to condition ($*$) (see \cite {Se} for a more detailed discussion and related facts).

 We denote by $A_+$ (resp. $A_-$) the set of points in $D \setminus A$ where $h_A+k$ is non-negative (resp. non-positive). Then, for any open set $O$, 
\begin {eqnarray*}
 \E\left[  (\Gamma+k,\1_O)^2 \I{O \subset A_+} \right] &=& \E \left[  \E \left[\left ( (h_A+k ,\1_O) + (\Gamma^A , \1_O)\right )^2 | {\mathcal F}_A \right]  \I{O \subset A_+}\right]  \\
 &=&  
\E\left[ \left( (h_A+k, \1_O)^2 + \int_{O \times O} G_{D \setminus A} (x,y) \d x \d y \right) \I{O \subset A_+}\right].
\end {eqnarray*}
Similarly, by conditioning on ${\mathcal F}_{A \cup B}$ and using that $A \cup B$ is thin, one gets that this same quantity is equal also to 
$$ \E\left[  \left( (h_{A \cup B}+k , \1_O)^2 + \int_{(O \setminus B)  \times (O \setminus B)} G_{D \setminus (A \cup B)} (x,y) \d x \d y\right)\I{O \subset A_+} \right] .$$
But by definition, when $O \subset A_+$, $(h_A+k, \1_{O}) \ge |(h_{A \cup B}+k , \1_O)|$ (because $h_A (x) +k  \ge |h_{A \cup B} (x) +k|$ for all $x \in A_+$). Hence, using the fact that $G_{D \setminus (A \cup B)} \le G_{D \setminus A}$, we conclude that 
for every open set, almost surely on the event $O \subset A_+$, 
$G_{D \setminus (A \cup B)} = G_{D \setminus A}$ on $O \times O$. The same statement holds for $A_-$ instead of $A_+$. Therefore $B \setminus A$ is polar in $D \setminus A$ and condition (3') allows us to conclude.
\end{proof}

Let us now suppose that $A$ is a local set of the GFF $\Gamma$ in a bounded simply-connected domain $D$ such that $D\backslash A$ is connected and $\partial D \cup A$ has only finitely many connected components.
The following lemma says that if $|h_A| \le C$ in the neighbourhood of all but finitely many prime ends of $D\backslash A$, then it is bounded by $C$ in all of $D\backslash A$. 
To state this rigorously, it is convenient to note that by Koebe's circle domain theorem, one can use a conformal map $\phi$ to map 
$D\backslash A$ onto a circle domain $\tilde O$ (i.e., the unit disc with a finite number of disjoint closed discs removed).
In this way, each prime end of $D\backslash A$ is in one-to-one correspondence with a 
boundary point of  $\tilde O$. Define also  the harmonic function 
$\tilde h_A := h_A \circ \phi^{-1}$ in $\tilde O$. 

\begin{lemma}\label{finite points}
Let $A$ be a local set of the GFF $\Gamma$ as just described, and let $\tilde O$, $\tilde h_A$ be as above. 
Assume furthermore that there exist finitely many points $y_1, \ldots, y_n$ on $\partial \tilde O$ and a non-negative constant $C$ such that for all $y \in \partial \tilde O \setminus \{ y_1, \ldots, y_n \}$,
one can find a positive $\eps(y)$ such that $|\tilde h_A | \le C$ in the $\eps(y)$-neighbourhood of $y$ in $\tilde O$.
Then  $|h_A|$ is in fact bounded by $C$ in all of $D\backslash A$.
\end{lemma}
\begin{proof} For some (random) small enough $r_0$, all connected components of $\partial \tilde O$ are at distance at least $r_0$ from each other.
Let us now consider any $\eps$ smaller than $r_0/2$. By compactness of $\partial \tilde O $, one can cover $ \partial \tilde O \setminus \cup_{j \le n} B(y_j, \eps)$
by a union $U$ of finitely many open balls of radius not larger than $\eps$ that are centered on points of $\partial \tilde O$ in such a way that $|\tilde h_A| \le C $ in 
all of  $U \cap \tilde O$.

Let us now choose some $\tilde z \in \tilde O$ with $d(\tilde z, \partial \tilde O) > \eps$ and prove that $|\tilde h_A (\tilde z) | \le C$. 
Define $V$ to be the connected component of $ \tilde O  \setminus ( U  \cup \cup_{j \le n} B(y_j, \eps))$ that contains $\tilde z$. 
The definition of $U$ shows that except on the (possibly empty) 
part of $\partial V$ that belongs to the boundary of the $\eps$-balls around $y_1, \ldots, y_n$, the function $| \tilde h_A |$ is bounded by $C$. 
Now, $\tilde h_A ( \tilde z)$ is the integral of the harmonic function $\tilde h_A$ with respect to the harmonic measure at $\tilde z$ on $\partial V$. Thus, in order to show that $|\tilde h_A (\tilde z)| \le C$, it suffices to prove that the contribution $J$ of the integral on $\partial V \cap \partial B (y_i, \eps )$ goes to $0$ as $\eps $ to $0$ for all $i=1 , \ldots , n$ .

To justify this, we can first note that the density of the harmonic measure (with respect to the Lebesgue measure) on all these arcs is bounded by a positive constant independently of $\eps$. 
On the other hand, it follows from the proof of Lemma 3.1 in \cite{HMP} that 
there exists a random constant $C'$ such that almost surely, the absolute value of the circle average of $\Gamma$ on the circle of radius $r$ around $x$ 
	is bounded by $ C'   \log(2 / r) $
	for all $x \in D$ and $r \in (0, 1]$ simultaneously. 	
As for all $x\in D\backslash A$, $h_A (x)$ is equal to the average of $\Gamma-\Gamma^A$ on any circle of radius smaller than $d(x,  \partial (D\backslash A) )$ around $x$, we deduce that  
	 for some random constant $C''$ and for all $x \in (D\backslash A)$
	$$ |h_A(x)| \le C'' \log (  2  / d(x, \partial (D\backslash A))  ). $$
But now Beurling's estimate allows to compare $d(x, \partial (D\backslash A))$ with $d(\phi (x), \partial \tilde O)$. 
We obtain that for some random positive $C'''$ and for all $y \in \tilde O$ simultaneously
$$  |\tilde h_A (y)| \le C''' \log ( 2 / | d(y, \partial \tilde O) | ).$$ 
This in turn implies that $J$ is almost surely $O(\eps|\log \eps|)$ as $\eps \to 0$, which completes the proof. 
\end{proof}

\section{Absolute continuity for the GFF, generalized level lines.}

\subsection {GFF absolute continuity}\label{ssabscon}

Let $D = \H \setminus K$, where $K$ is a countable union of closed discs such that in any compact set of $\H$ there are only finitely many of them.  

Let us recall first that, similarly to Brownian motion, the GFF can be viewed as the Gaussian measure associated to the Dirichlet space $\mathcal{H}^1_0$, which is the closure of the set of smooth functions of compact support in $D$ with respect to the Dirichlet norm given by 
$$(f,f)_\nabla = \int_{D} |\nabla f(x)|^2  \d x . $$
The Dirichlet space is also the Cameron-Martin space for the GFF (see e.g. \cite{Dub, ShQZ} for this classical fact): 
	\begin{thm}[Cameron-Martin for the GFF]
		Let $F$ be a function belonging to $\mathcal H^1_0(D)$ and $\Gamma$ a GFF in $D$. Denote the law of $\Gamma$ by $\P$ and the law of $\Gamma+F$ by $\tilde \P$. 
		Then $\P$ and $\tilde \P$ are mutually absolutely continuous and the Radon-Nikodym derivative $d\tilde \P / d \P$ at $\gamma$ is a multiple of $\exp((F, \gamma)_\nabla)$.
	\end{thm}

	We are now going to use this in the framework of local sets of the GFF:  
{Denote by $S$ the interior of a} finite union of closed dyadic squares in $D$ with $I:= \partial S \cap \R \not= \emptyset$,
and a harmonic function $H$ in $D$ such that $H$ extends continuously to an open neighbourhood $I'$ of $I$ in $\R$ in such a way that $H =0$ on $I'$ (we then say that the boundary value of $H$ on $I$ is zero). Using the Cameron-Martin Theorem it is not hard to see that the GFF $\Gamma$ and  $\Gamma + H$ are mutually absolutely continuous, when restricted to $S$, i.e. when restricted to all test functions $f$ with support in $S$. 

Indeed, let $\tilde H$ be the bounded harmonic function in $D\backslash\partial S$ that is equal to $H$ on the boundary of $S$ and to zero on the boundary of $D$. Note that $\tilde H$ belongs to $\mathcal H^1_0(D)$ and that $(\tilde H, \Gamma)_\nabla$ depends only on the restriction of $\Gamma$ to $S$. Thus, using the Cameron-Martin Theorem and the domain Markov property of the GFF we obtain (see \cite{Dub,ShQZ}):

\begin{lemma}\label{abscnty}
Let $\Gamma$ be a (zero boundary) GFF in $D$ and denote its law restricted to $S$ by $\P$. Let also $\tilde \P$ be the law of $\Gamma + H$, restricted to $S$. Then $\P$ and $\tilde \P$ are mutually absolutely continuous with respect to each other. Moreover, the Radon-Nikodym derivative $d\tilde \P / d \P$  is a multiple of $\exp((\tilde H, \Gamma)_\nabla)$. 

Furthermore, if $D'$ is another domain as above, $\Gamma'$ is a GFF in $D'$ and $S\subset D\cap D'$ is at positive distance of $\partial D\triangle \partial D'$, then the laws of $\Gamma$ and $\Gamma'$ restricted to $S$ are mutually absolutely continuous.
\end{lemma}

This absolute continuity property allows to change boundary conditions for local couplings away from the local sets:
{\begin{prop}\label{change boundary} 
	Under the previous conditions, suppose that $A$ is a BTLS for $\Gamma$ an $D$-GFF such that $A \subset S\subset D$ almost surely. Define $\tilde \P:= Z \exp((\tilde H, \Gamma)_\nabla)d\P$. Then $A$ is coupled as a local set with $\tilde \Gamma:=\Gamma-\tilde H$ (that is a $\tilde \P$ zero-boundary GFF). The corresponding harmonic function $\tilde h_A$ is equal to the unique bounded harmonic function on $D \setminus A$, with boundary values equal to those of $h_A-H$ on $\partial A \cup I$, and to $0$ on $\partial D\setminus I$. 
\end{prop}}

\begin{proof}	
	Notice that by the Cameron-Martin Theorem for the GFF, the law of $\Gamma$ under $\tilde{\P}$ is that of $\tilde\Gamma + \tilde H$, where $\tilde \Gamma$ is a $\tilde \P$ GFF. Now we have to verify that with this change of variables $A$ is still a local set. Note that conditionally on ${\mathcal F}_A$,
	we can write $\Gamma=h_A+\Gamma^A$, where $\Gamma^A$ is a $\P$-GFF on $D\setminus A$. We have to show that  we can write $\Gamma^A= \tilde \Gamma^A +\tilde H+ \tilde h_A - h_A$, where $\tilde \Gamma^A$ is a $\tilde \P$-GFF on $D\setminus A$ and $\tilde h_A$ is as in the statement.
	Let $M$ be a measurable function of the field $\Gamma^A$, {then for some  $Z_A$ and $Z_A'$, measurable functions of $A$,}
	\begin{align*}
	\tilde \E\left[ M(\Gamma^A)\mid A,h_A \right] &=Z_A \E\left  [ M(\Gamma^A) \exp\left ( (\tilde H,\Gamma)_{\nabla} \right )\mid {\mathcal F}_A \right ] \\
	&=  Z_A' \E\left [ M(\Gamma^A) \exp\left ( (\tilde H,\Gamma^A)_{\nabla} \right )\mid { \mathcal F}_A \right ]\\
	&=  Z_A' \E \left[  M(\Gamma^A) \exp\left ( (\tilde H+\tilde h_A - h_A,\Gamma^A)_{\nabla} \right )\mid {\mathcal F}_A \right],
	\end{align*}
	where in the last equality we use that $\Gamma^A$ is a GFF and that $\tilde h_A - h_A$ is harmonic in $D \setminus A$. But, we know that (under $\P$), conditionally on ${\mathcal F}_A$, $\Gamma^A$ is just a GFF. {Thus using again the Cameron-Martin Theorem, Lemma \ref{BPLS}-(i) and the fact that $\tilde h_A-h_A$, $\tilde H$ are in $\mathcal H^1(D\backslash A)$ we can conclude.} 
\end{proof}

 {
\begin{cor}\label{change boundary2} 
Let $D'\subseteq D$ be another domain with the same properties. Under the previous conditions, suppose that $A$ is a BTLS for a GFF in $D$ such that $A \subset S\subset D'$ almost surely. Then, there exist an absolutely continuous probability measure $\tilde \P$ under which $A$ is coupled as a local set with a $\tilde \P$ zero-boundary GFF in $D'$ and the corresponding harmonic function $\tilde h_A$ is equal to the unique bounded harmonic function on $D' \setminus A$, with boundary values equal to those of $h_A-H$ on $\partial A \cup I$, and to $0$ on $\partial D'\setminus I$. \end{cor}}

\begin{proof}	
 It can be shown that $A$ is a local set for $\Gamma^{D\backslash D'}$ such that its harmonic function $(h^{D\backslash D'})_A$ goes to 0 on $\overline \partial D \setminus I$ and $(h^{D\backslash D'})_A-h_A+h_K$ goes to $0$ on $\partial A \cup I$ (see Lemma 3.9-1 of \cite{SchSh2}, or \cite{WWln2}). We conclude using the fact that $h_{D\backslash D'}$ is independent of $h^{D\backslash D'}$ and Proposition \ref{change boundary}.
 \end{proof}

\subsection {SLE$_4$ and the GFF with more general boundary conditions}
\subsubsection* {SLE$_4$ as level lines of the GFF}
Let us start by recalling some well-known features of the Schramm-Sheffield coupling of SLE$_4$ with the GFF in the upper half-plane $\HH$:
{
Consider the bounded harmonic function $F_0  (z)$ in the upper-half plane with boundary values $-\lambda$ on $\R_-$ and $+\lambda$ 
on $\R_+$. There exists a unique law on random simple curves  $(\eta(t), t \geq 0 ) )$ (parametrized by half-plane capacity) in the closed upper half-plane from $0$ to infinity that can be coupled with a GFF $\Gamma$ so
that the following property holds for all $t \geq 0$. (we state it in a somewhat strange way that will be easier to generalize):
\begin{description}
\item [$(**)$] The set  {$\eta[0, t]$} is a BTLS of the GFF $\Gamma$, with harmonic function $h_t$ defined as follows: $h_t + F_0$ is the  unique bounded harmonic function 
in $\H \setminus \eta[0,\min \{t, \tau\}]$ with boundary values $-\lambda$  on the left-hand side of $\eta$, $+ \lambda$ on the right side of $\eta$, and with the same boundary values as $F_0$ on $\R_-$ and $\R_+$. 
\end{description}
Furthermore, this curve is then a SLE$_4$ and when one couples SLE$_4$ with a GFF $\Gamma$ in this way, then the SLE$_4$ process 
is in fact a measurable function of the GFF (see \cite {Dub,MS1,SchSh2,ShQZ,WWln2} for all these facts). }

This statement can be generalized to more general harmonic functions $F$, which gives rise to SLE$_4 (\rho)$-type processes. See for instance \cite{PoWu}.

\subsubsection* {More general boundary conditions}
{Now we generalize this definition of level lines to the GFF with more general boundary conditions. By conformal invariance if we wish to define them for all domains $D$ described at the beginning of Section \ref{S3} it is enough to consider the case where  $D = \H \setminus K$, where $K$ is a countable union of closed discs such that any compact subset of $\H$ intersects only finitely many of these discs. 
	
	Let $H$ be a  harmonic function on $D$ with zero boundary conditions on some real neighbourhood $I$ of the origin. For a random simple curve $\eta$ in $D$ we define time $\tau$ to the (possibly infinite) smallest positive time $t$ at which $\liminf_{s \to t-} d( \eta_s, \partial D) = 0$. The generalized level line for the GFF in $D$ with boundary conditions $F_0 + H$ up to the first time it hits the boundary is then defined as follows:

\begin {lemma}\label{defll} [Generalized level line]  There is a unique law on random simple curves  $(\eta(t), t \geq 0 ) )$  in $D$ parametrized by half-plane capacity (i.e., viewing $\eta((0,t])$ as a subset of $\H$ instead of $\H \setminus K$) with 
$\eta(0) = 0$, $\eta (0, \tau ) \subset D$ that can be coupled with a GFF $\Gamma$ so
that $(**)$ holds  for all $0 \leq t < \tau$,
 when one replaces $F_0$ by $F_0 + H$ and considers $D$ instead of $\H$. Moreover, the curve $\eta$ is measurable with respect to $\Gamma$. We call $\eta$ the generalized level line of $\Gamma+F_0+H$ in $D$.
\end {lemma}}

\begin{proof}
	Let $\mathcal{S}$ denote the collection of regions $S$ such that $S$ is the interior finite union of closed dyadic squares in $D$ with $0 \in  \partial S \cap \partial D \subset I$ and with $S$ simply connected. Note that is enough to show that for all $S\in \mathcal S$ there is at most one curve satisfying $(**)$, when one replaces $F_0$ by $F_0 + H$, until the time it exits $S$. 
	
	Suppose by contradiction that $\eta_1$ and $\eta_2$ are two different curves with this property and that with positive probability they do not agree. Thanks to Corollary \ref{change boundary2}, we can construct two local sets $\tilde \eta_1$ and $\tilde \eta_2$ coupled with the same GFF in a simply connected domain $D'\subset D$ that strictly contains $S$, such that when we apply $\phi$, the conformal transformation from $D'$ to $\H$, $(**)$ holds for $\phi(\tilde \eta_1)$ and $\phi(\tilde \eta_2)$ until the time they exit $\phi(S)$. The fact that they are different with positive probability contradicts the uniqueness of the Schramm-Sheffield coupling proved in \cite{Dub,SchSh2}.
	
	It remains to show the existence of this curve. Take $S_n$ an increasing sequence of elements in $\mathcal S$ such that their union is $D$. We can define $\eta$ until the time it exits $S_n$, by using the Schramm-Sheffield coupling and Proposition \ref{change boundary2}. The above argument shows that they are compatible and so we can define $\eta$ until the first time it touches $\partial D$. The measurability is just a consequence of the existence and uniqueness.
\end{proof}

\begin{figure}[ht!]	   
	\centering
	\includegraphics[scale=0.39]{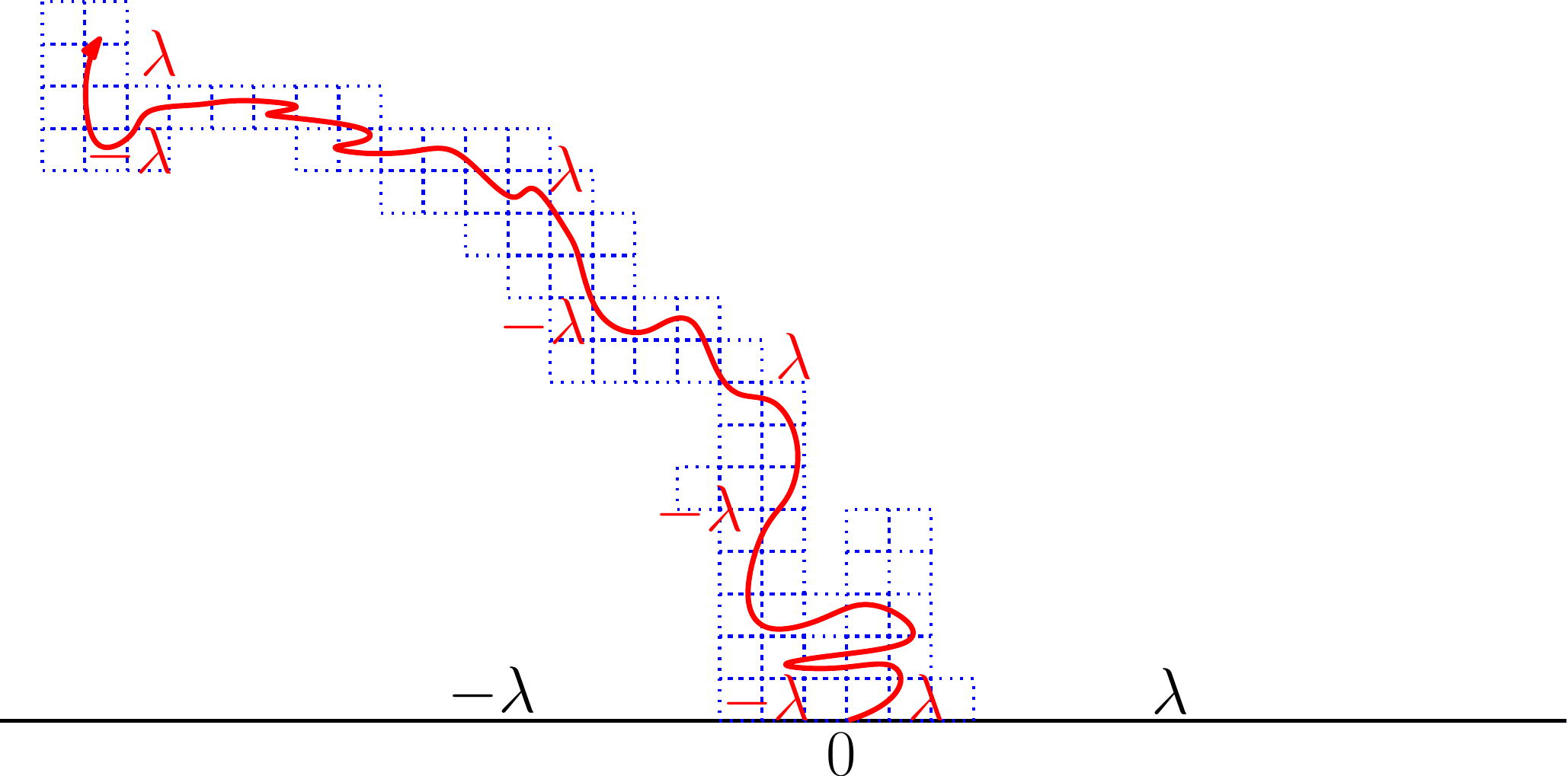}
	\includegraphics[scale=0.39]{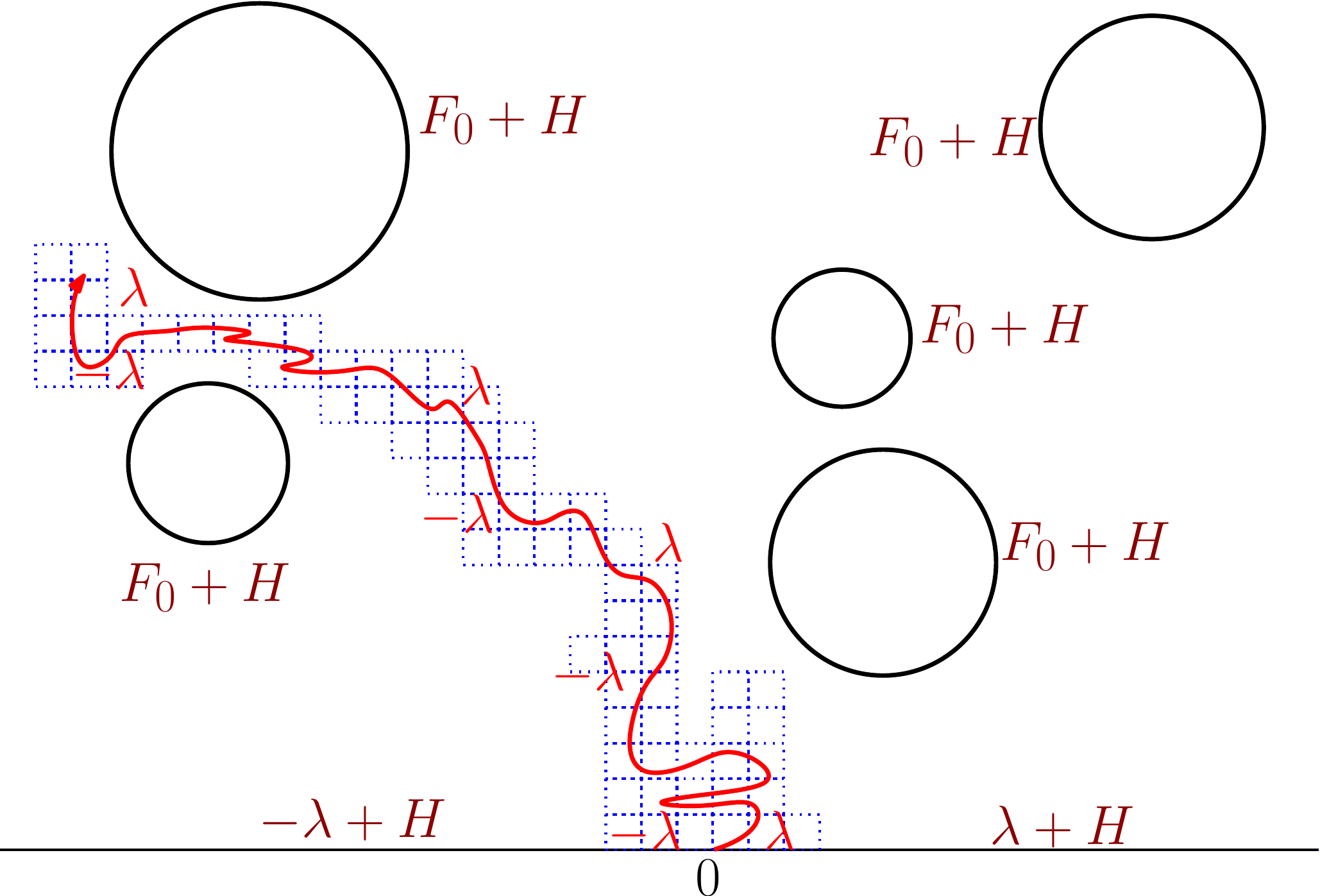}
	\caption {Defining generalized level lines by absolute continuity.}
	\label {abscon}
\end{figure}
{Note that for all stopping times $\tau$ of $\eta([0,t])$, such that $\tau$ is measurable with respect to the $\sigma$-algebra generated by $\eta$, we have that ($**$) holds for $t=\tau$.}

\subsubsection*{Boundary hitting of generalized level lines}
In \cite{PoWu} Lemma 3.1, the authors show that a generalized level line in $\H$ corresponding to $F_0 + H_0$ with $H_0$ zero in a neighbourhood of the origin, bounded, $F_0 + H_0 \geq -\lambda$ on $(-c,0)$ and $F_0 + H_0 \geq \lambda$ elsewhere cannot touch nor accumulate in a point of $\R \setminus [-c,0]$. Combined with our previous considerations, we extend this to:
  {
 \begin{lemma}[Boundary hitting of generalized level lines] \label{notouch}
 	Let $\eta$ be a generalized level line of $\Gamma + F_0 + H$ in $D$ such that $H$ is harmonic in $D$ and has zero boundary values in a neighbourhood $I$ of zero.{ Suppose $F_0 + H \geq \lambda$ in $J\cap \partial D$, with $J$ some open set of  $D$}. Let $\tau$ denote the first time at which $\inf \{ d ( \eta_s, J\cap \partial D), s< t \} = 0$. Then, the probability that $\tau < \infty$ and that $\eta(t)$ converges to a point in $J$ as $t \to \tau^-$ is equal to $0$. This also holds if $F_0 + H \leq -\lambda$ in $J\cap \partial D$ for $J$ an open set of $D$.
 \end{lemma}}

\begin{proof}
{Let $S$ be the interior of a finite union of closed dyadic squares intersected with $D$ such that $0 \in  \partial S$, $x \in \partial S$, $\partial S \cap \partial D \subset I \cup J$ and $S \cap K = \emptyset$. From Lemma \ref{defll} we know that for any such $S$, a generalized level line $\eta$ is measurable up to leaving $S$ w.r.t. the GFF restricted to $S$. 

We first consider the case $ x<0$ and $J\cap \partial  D= (x-\epsilon,x+\epsilon)$ with $0<\epsilon< -x$. Define a harmonic function $H_0$ on $\H$ such that $H_0 = H$ on $(x-\eps,x+\eps)$, $H_0 = 0$ on $[x+\eps,\infty)$ and $H_0 = 2\lambda$ on $(-\infty,x-\eps]$. The laws of the GFFs $\Gamma + F_0 + H_0$ and $\Gamma + F_0 + H$ restricted to $S$ are absolutely continuous by Lemma \ref{abscnty}. Also, both of these generalized level lines are measurable w.r.t. the GFF until they exit $S$. Now $F_0 +H_0$ satisfies the conditions mentioned above and hence almost surely the generalized level line of $\Gamma + F_0 + H_0$ does not exit from $S$ on $\partial S \cap (x-\eps,x+\eps)$ in finite time. By absolute continuity and measurability of both generalized level lines, the result follows for this $J$. The case $x > 0$ is treated similarly. By the union bound the result is true simultaneously for all $x_n \in \Q$ and $\epsilon_n \in \Q$ with $\eps_n < |x_n|$ and $F_0 + H \geq \lambda$ on $(x_n-\eps_n, x+\eps_n)$. 
Given that 
all $J\cap \partial D\subseteq \R$ satisfying the conditions of the theorem are written as the union of these intervals, we conclude.

 To treat the case where $x$ is on the boundary of some component $D_1$ of $K$, notice that it suffices to consider $S$ that do not surround $D_1$ and that for any such $S$ we can connect $D_1$ to $\R$ using some curve $\gamma$ such that $\H \setminus (D_1 \cup \gamma)$ is simply connected and contains $S$. The previous argument and Corollary \ref{change boundary2} then help to conclude.}
\end{proof}

\section{Review of the construction of CLE$_{4}^M$ and its coupling with the GFF}
\label{CLE4}

In this section,  we review the coupling of the GFF with the nested CLE$_4$ using Sheffield's SLE$_4$ exploration trees (i.e. the branching SLE$_4(-2)$ process). 
 As this section does not contain really new results, we try to be rather brief, and refer to \cite {She,WWln2} for details. 
 Note that we will present an alternative construction of the CLE$_4$ and its coupling to the GFF in Section \ref {S6}. 

\subsection {Radial symmetric SLE$_4 (-2)$}
Let us first recall the definition of the radial (symmetric) SLE$_4 ( -2)$ process targeted at $0$ (for non-symmetric variants see \cite{She}). It is the radial Loewner chain {of closed hulls} $(K_t)_{t \ge 0}$ in the unit disc, with driving function $\xi_t$ defined as follows:
\begin{itemize}
\item Start with a standard real-valued Brownian motion $(B_t, t \ge 0)$ with $B_0 = 0$. 
\item Then, define the continuous process $U_t := \int_0^t \cot (B_s ) \d s$.  
\item Finally, set $W_t = 2B_t+U_t$ and $\xi_t := \exp (i W_t) $. 
\end{itemize}
We set $D_t = \D \setminus K_t$,  and denote by $f_t$ to be the conformal maps from $D_t$ onto $\D$ normalized at the origin.

Notice that as the integral $\int_0^t |\cot (B_s )| ds$ is almost surely infinite, the definition of $U_t$ can not be viewed as an usual absolutely converging integral. Yet, one can for instance define first $U_t^\eps$ as the integral of $\I{d( B_s, \pi \Z) \ge  \eps } \cot (B_s)$ and show that these approximating processes $U_t^\eps$ converge in $L^2$ to a continuous processes $U_t$. See e.g. \cite{WW} for a discussion of the chordal analogue.

We now describe some properties of the hulls generated by the radial SLE$_4 (-2)$. For more details, see \cite {She,Wln,WW}: 
\begin {enumerate}
\item 
The process $O_t := \exp ( i U_t)$ describes the evolution of the marked point in the SLE$_\kappa ( \rho )$ framework. In particular, SLE coordinate change considerations (see \cite {Dub0,SchShW})  show that during the time-intervals 
during which $B_t$ is not in $\pi \Z$, the process behaves exactly like a chordal SLE$_4$ targeting the boundary prime end $f_t^{-1} (O_t)$ in the domain $D_t$.
\begin{figure}[ht!]
\begin{center}
\includegraphics[width=3.1in]{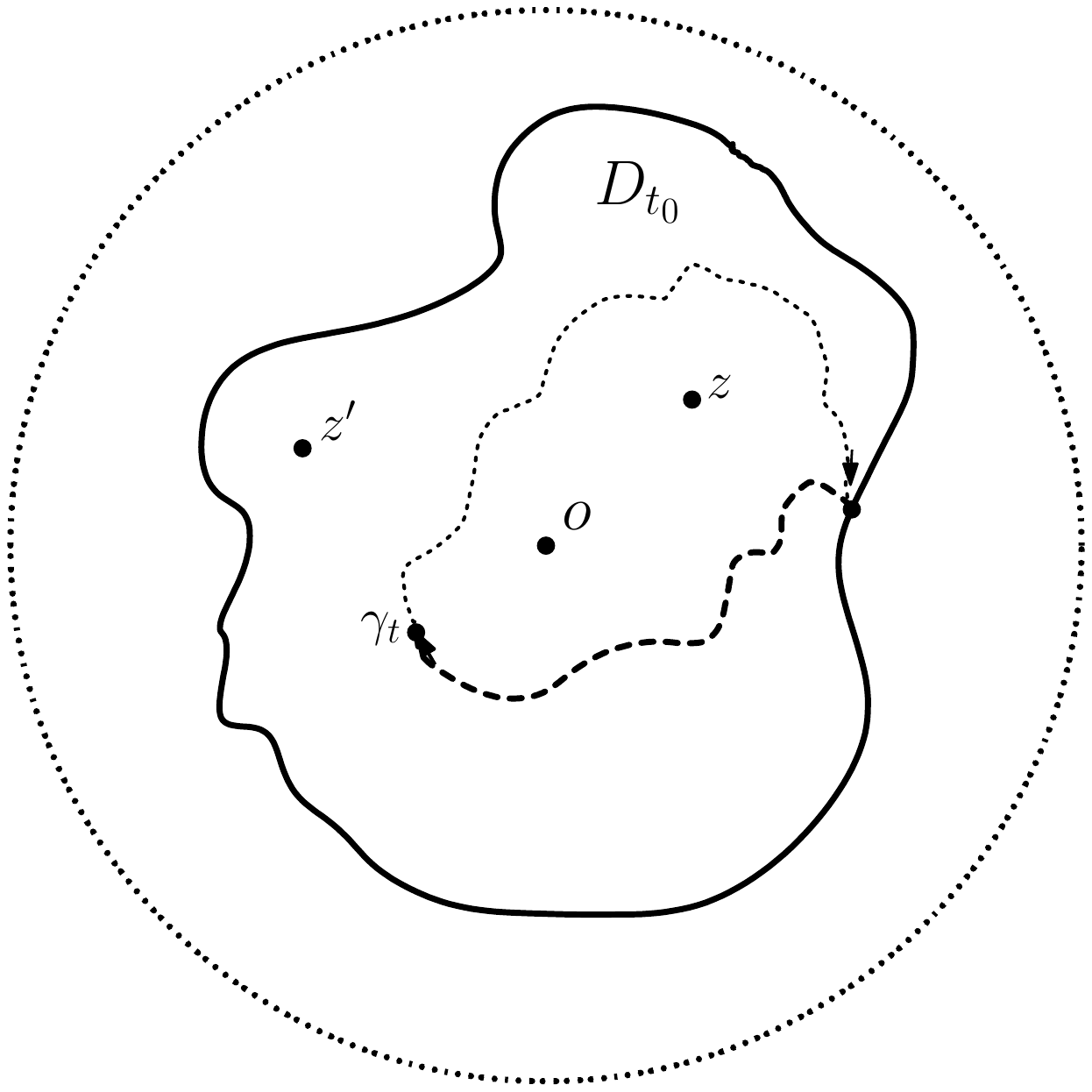}
\caption{\label{figDt}Sketch of $D_{t_0}$, the stretch $\gamma(t_0, t]$. The harmonic measure from $0$ of the ``top part'' of $\gamma[t_0, t]$ corresponds to $B_t$. This loop corresponds to an excursion away from $\pi \Z$ by each of $B$, $B^z$ and $B^{z'}$. 
It corresponds to an increase of $\pi$ for $B$ and for $B^{z}$, but 
not for $B^{z'}$. The end-time of this loop is $\sigma (z,z')$.}
\end{center}
\end{figure} 

\item
It is therefore clear that during the open time-intervals at which $B$ is not in $\pi \Z$,
the Loewner chain is generated by a simple continuous curve. These excursions of $B$ away from $\pi \Z$ correspond to the intervals during which the SLE$_4 (-2)$ traces ``quasi-loops'': If $(t_0, t_0')$ is an  excursion interval away from $\pi \Z$, 
then $\gamma (t_0, t_0')$ is the image of a loop from $\xi_{t_0}$ to itself in $\D$ via $f_{t_0}^{-1}$. 
When $B_{t_0}= B_{t_0'}$, 
then this quasi-loop does not surround the origin, when $B_{t_0'} = B_{t_0} \pm \pi$, it surrounds the origin clockwise or anti-clockwise. We stress that during any time-interval corresponding to an excursion of $B$ away from $\pi \Z$ {(e.g during $(t_0, t_0')$)}, the tip of the Loewner chain does not touch its past boundary and at the end of the excursion it accumulates at the prime end corresponding to the starting point.
\item { It is known (via the loop-soup construction of CLE$_4$, see \cite {ShW}) that these quasi-loops are in fact proper loops (i.e. that $f_{t_0}^{-1} (\xi_{t_0})$ is a proper boundary point). It can be also shown that the whole SLE$_4 (-2)$ is {\em generated} by a two-dimensional curve, see \cite {MSW}, but we do not need this fact in the present paper. }
\end {enumerate}

 For any other $z \in \D$, one can define the radial SLE$_4 (-2)$ from $1$ to $z$ in $\D$ to be the conformal image of the  radial SLE$_4 (-2)$ targeting the origin, by the Moebius transformation of the unit disc that maps $0$ to $z$, and $1$ onto itself. This is now a Loewner chain growing towards $z$, and it is naturally parametrized using the log-conformal radius seen from $z$.
 
 Moreover the radial SLE$_4(-2)$ processes for different target points can be coupled together in a nice way. This target-invariance feature was first observed in \cite {Dub0,SchShW}, and is closely 
 related to the decompositions of SLE$_4 (-2)$ into SLE$_4$ ``excursions'' mentioned in (1). It says that for any $z$ and $z'$, the SLE$_4 (-2)$ targeting $z$ and $z'$ can be coupled in such a way that 
they coincide until the first time $\sigma(z,z')$ at which $z'$ gets disconnected from $z$ for the chain targeting $z$ and evolve independently after that time. Note that in this coupling, the natural time-parametrizations of these two processes do not coincide. However, using the previous observation about the relation between the excursions of $B$ and the intervals during which the SLE$_4 (-2)$ 
traces a simple curve, we see that up to time $\sigma(z,z')$ the excursion intervals away from $\pi \Z$ by the Brownian motions used to generate the SLE$_4(-2)$ processes targeting $z$ and $z'$ correspond to each other. The time-change between the two Brownian motions can be calculated explicitly. For instance, if we define $B_t^z$ to be the process obtained by time-changing the Brownian motion used in the construction that targets $z$ via the log-conformal radius of $D_t$ seen from $0$ (instead of $z$), then as long as $t < \sigma (z, 0)$,
 $$d B_t^z := 2\pi P^{f_t(z)} dB_t$$
where by $P^z$ we denote the Poisson kernel in $\D$ seen from $z$ and targeted at $\xi_t$ (it follows from the Hadamard formula given below).

\subsection {SLE$_4 (-2)$ branching tree and CLE$_4$}
 
The target-invariance property of SLE$_4(-2)$ leads to the construction of the radial SLE$_4 (-2)$ branching-tree. We next summarize its definition and its main properties: 
\begin {itemize}
 \item For a given countable dense collection $\mathcal {Z}$ of points in $\D$, it is possible to define on the same probability space a collection of radial SLE$_4 (-2)$ processes in such a way that for any $z$ and $z'$, they coincide (modulo time-change) up to the first time at which $z$ and $z'$ get disconnected from each other, and behave independently thereafter. 
 \item For any $z$ consider the first time $\tau_z$ at which the underlying driving function driving function exits the interval $(- \pi, \pi)$ and define $O(z)$ to be the complement of $K_t$ at that time containing $z$. Define $O(z')$ similarly. The previous property shows that $z' \in O(z)$ if and only if $z \in O(z')$. 
\item 
The CLE$_4$ carpet is then defined to be the complement of $\cup_{z \in \mathcal{Z}} O(z)$. In this article, we write CLE$_4$ to denote this set (i.e. we call CLEs to be the random fractal sets, not the collection of loops). 
 \end {itemize}
 
Notice that from the construction, it is not obvious that the law of the obtained CLE$_4$ (and its nested versions) does not depend on the choice of the starting point (here $1$) on $\partial \D$ (i.e. of the root of the exploration tree). However, Proposition \ref{cledesc} and Proposition \ref{cledesc2} provide one possible proof of the fact that this choice of starting point does not matter (this is also explained in \cite {WW}, using the loop-soup construction of \cite {ShW}). 
This  justifies a posteriori that when one iterates the CLE$_4$ in order to construct the nested versions, one does in fact not need to specify  where to continue the exploration. 
 
On the other hand, from this construction it is easy to estimate the expected area of the $\eps$-neighbourhood of CLE$_4$ and to see that the upper Minkowski dimension of CLE$_4$ is almost surely not greater than $15/8 = 2-1/8$. This follows from arguments in \cite{SchShW}:  One wants to estimate the probability that $d( z, \CLE_4 ) < r$. By conformal invariance it is enough to treat the case where $z$ is the origin in the unit disc. But we can note that the conformal radius of $D_t$ (in the SLE$_4 (-2)$) is comparable to $d (0, \partial D_t)$ by Koebe's $1/4$-Theorem, and 
that the log-conformal radius of $\D \setminus \CLE_4$ from the origin is just minus the exit time of $(- \pi, \pi)$ by the Brownian motion $B$ which is a  well-understood random variable (see \cite {PS} for instance). Hence, the asymptotic behaviour as $r \to 0$ of the probability that $d(0 , \CLE_4) < r$ can be estimated precisely. 

One can also prove the somewhat stronger statement that the Hausdorff dimension of CLE$_4$ is in fact equal to $15/8$, by using second moment bounds (i.e., bounds on the probability that two points $z$ and $z'$ are close to CLE$_4$), see \cite {NW}.

The nested CLE$_{4,m}$ and CLE$_4^M$ can then be defined by appropriately iterating independent CLE$_4$ carpets in the respective domains $O(z)$ (starting the explorations in the nested domains at the point 
where one did just close the loops). The definition of CLE$_4^M$ is almost word for word the same, just replacing $\tau(z)$ by $\tau_M (z)$ which is the first time at which the underlying Brownian motion $B$ exits the interval $(-M\pi, M\pi)$. A similar argument then shows that the upper Minkowski dimension of CLE$_4^M$ is not greater than $2 - (1 / (8M^2))$ (the $M^2$ term is then just due to Brownian scaling -- the exit time of $(- M \pi, M \pi)$ by Brownian motion is equal in distribution to $M^2$ times the exit time of $(- \pi, \pi)$).

\subsection{Coupling with the GFF}

To explain the coupling of CLE$_4$ with the GFF, we can first describe the coupling of a single radial SLE$_4(-2)$ process with a GFF. The whole coupling then follows iteratively from the strong Markov property and from the branching tree procedure.
The proof of the coupling of the radial SLE$_4 (-2)$ follows the steps of the coupling of the usual SLE$_4$ with the GFF, as explained for instance in \cite {SchSh2, ShQZ}. It is based on the following observations:

\begin{enumerate}
\item The Hadamard formula (see for instance \cite{IzKy}) gives the time-evolution of the Green's function under Loewner flow: for any two points $x$ and $y$ in $\D$, the Green's function $G_{D_t} (x,y)$ evolves until $\min\{ \sigma (z, 0), \sigma (w,0)\}$
according to 
$$ d G_{D_t} (z,w) = -  2\pi P^{f_t(z)}P^{f_t(w)} dt,$$
where by $P^z$ we denote the Poisson kernel in $\D$ seen from $z$ and targeted at $\xi_t$ as before.
\item The cross-variation  $d \langle B^z, B^w\rangle_t$ between the two local martingales $B^z$ and $B^w$ is equal to $4\pi^2 P^{f_t(z)}P^{f_t(w)} dt$, so that $$d G_{D_t}(z,w) = - \frac{1}{2\pi} d \langle B^z,B^w\rangle_t .$$
\item We can interpret $|B_t^z| / \pi $ (up to $\tau (z)$) as the harmonic measure in $D_t$ seen from $z$, of the boundary arc between the tip of the curve and the force point $U_t$ (the sign of $B_t^z$ describes which of the two arcs one considers).
\item {If we take $h_t(z)$ to be the harmonic extension to $D_t$ of the function that has constant value $\sgn (B_t)2\lambda$ on the boundary of $\partial D_t$ between the tip and the force point, then the mean of $h_t(z)$ is zero and for $\lambda = \sqrt{{\pi}/{8}}$ we have:
$$\d G_{D_t}(z,w) = - \d\langle h_t^z,h_t^w\rangle.$$}
\end{enumerate}
Using these observations, one can first couple the GFF with the radial SLE$_4 (-2)$ up to the first time at which it surrounds the origin, exactly following the arguments of \cite {ShQZ}. This defines a BTLS with harmonic function in $\{ - 2 \lambda, 0, 2 \lambda\}$. 
For those domains where the harmonic functions are zero, one can then continue the SLE$_4 (-2)$ iteration targeting another well-chosen point in that domain and proceed.  

This radial construction is arguably also the easiest one to explain that in fact, conditionally on the CLE$_4$, the labels $\eps_j$ are i.i.d. This is then just a consequence of the time-reversal of SLE$_4$, so that changing the orientation 
of the quasi-loop that traces the boundary of $O_j$ corresponds to a measure-preserving transformation of the driving Brownian motion (see for instance \cite {WW}). We come back to this in \ref {commenting}. 

\section {Comparisons of BTLS with CLE$_4^M$}

We  now derive Proposition \ref {maxbtls}, Proposition \ref {cledesc} and Proposition \ref {propnodisconnection}. In some sense, this section is the core of the paper. 

Throughout this section, $D$ will denote a simply connected planar domain with non-empty boundary (so that the previous definition of CLE$_4$ makes sense). 

\subsection{Proof of Proposition \ref{maxbtls}}

In this proof, $(C, h_C)$ will denote the CLE$^{M+1}_4$ and its corresponding harmonic function. 
Consider the triple $(\Gamma, (B, h_B), (C,h_C))$ where the two local sets $(C,h_C)$ and $(B, h_B)$ are conditionally independent given $\Gamma$ and where $B$ is a ($2\lambda M$)-BTLS. 
Recall from Section \ref{CLE4} that $(C, h_C)$ can be constructed using a branching radial SLE$_4$ exploration tree denoted by SLE$^r_4$. 

The two steps of the proof are then as follows: 
\begin{itemize}
	\item Given $(B, h_B)$, we argue using Lemma \ref{notouch} that when a branch of the SLE$^r_4$ exploration tree is in the process of tracing a loop of $C$ inside a connected component $O$ of $D\backslash A$, then the loop it generates has to be contained in $O$.
	\item From this we deduce that for all $z \notin B \cup C$, one has $|h_{B \cup C}(z)| \leq 2\lambda(M+1)$. We then conclude using Lemma \ref{mes}. 
\end{itemize}

\begin{figure}[ht!]
	\centering		
	\includegraphics[scale=0.8]{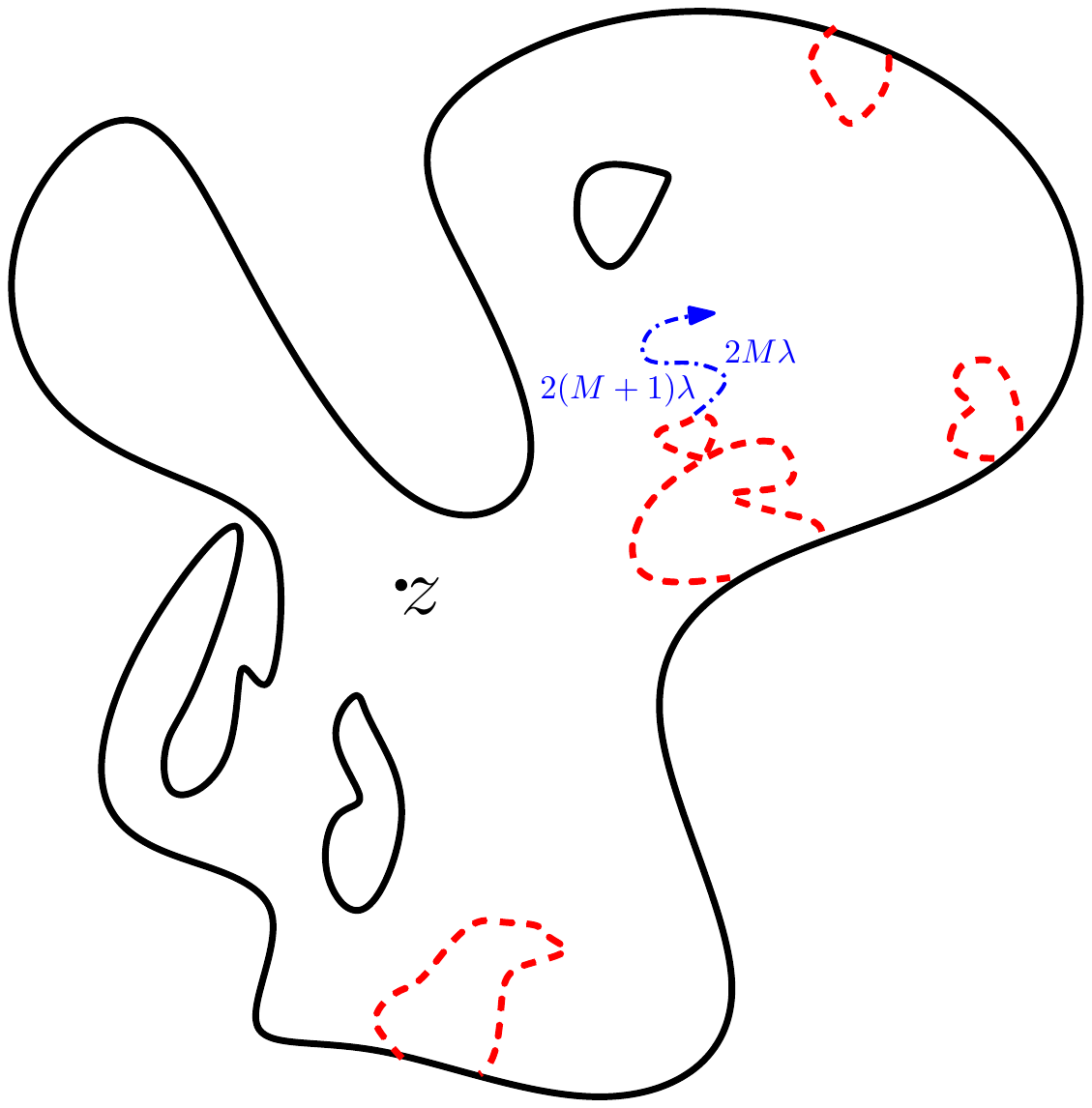}
	\caption{Sketch of the proof: in plain $\partial O_B(z)$; in dashed, the past of the radial exploration (used to define the CLE$_4^{M+1}$) that lies within $O_B(z)$. 
	When one explores a loop along the  dash-dotted interface: 
		It cannot hit the dashed part before completing the CLE$^{M+1}_4$ loop because of the almost sure properties of the radial SLE$_4 (-2)$ process.
		Also, it cannot touch the parts of $\partial O_B(z)$ that are away from the dashed part before completing the loop, because of the boundary hitting behaviour of the GFF level lines. }
	\label{fig:test}
\end{figure}

Let us now make it precise. For $z \in D$, we define by $O_B(z)$ (respectively $O_C (z)$) the connected component of $D\setminus B$ (resp. $D \setminus C$) containing $z$ when $z \notin B$ (resp. $z \notin C$).  Denote by $\nu^z$ the process obtained from the branch of the SLE$^r_4$ tree that is directed at $z$ and recall from Section \ref{CLE4} that $\nu^z([0,t])$ is a BTLS for all fixed $t$. Let $K^z_t$ be the hull of $\nu^z_t$ with respect to the point $z$, i.e. the complement of the open component of $D \setminus \nu^z[0,t]$ containing $z$.

\begin{claim}\label{loopsclose} Fix $t>0$ and $z\in D$. Define $D_t(z):= O_B(z) \setminus K^z_t$ and let $E_{t,z}$ be  event that 
 $\nu^z(t)$ is in the middle of tracing a loop of $C$ and that  $\nu^z(t) \in O_B(z)$.
Suppose that the  probability  of $E_{t,z}$ is strictly positive. 
Then, on the event $E_{t,z}$, the time  $\tau = \sup \{t': \nu^z[t,t'] \subset D_t(z)\}$ is finite and it is exactly the time at which the exploration closes the loop that it 
is tracing at time $t$. Thus, the starting and ending point of the loop correspond to at most two prime ends of $\partial D_\tau(z)$. 
\end{claim}

\begin{proof}[Proof of Claim \ref{loopsclose}]
Let us notice that $D_t(z)$ satisfies the condition of Lemma \ref{defll}. Thus, conditioned on $(B, h_B)$, $(\nu^z[0,t], h_{\nu^z[0,t]})$ and $E_{t,z}$ the process $\tilde \nu(s) :=(\nu^z(t+s), s \geq 0)$ is a generalized level line in the domain $D_t$ up to the time $\tau': = \sup \{s: \nu^z([t,t+s]) \subset D_t(z)\}$. On the event $E_{t,z}$, the path $\eta_z$ is locally tracing a level line with heights $\pm (2 \lambda(M+1))$ vs. $\pm (2 \lambda M)$ at time $t$.  

We know from Section \ref{CLE4} that almost surely the SLE$^r_4$ exploration process touches itself only when it closes a loop and stays at a positive distance of any other previously visited point.

It remains to show that $\tilde \nu_s$ does not touch any point of $\partial D_t\backslash K_t^z$. Take $J$ any open set of $D_t$ such that $d (J, K_t^z)\geq \epsilon$. From Lemma \ref{BPLS}, the boundary condition of  $h_{B \cup \nu^z[0,t]}$ in  $\partial J \cap \partial D_t$ are equal to those of $h_B$, thus their absolute value is not larger than $2M\lambda$. Lemma \ref{notouch} now lets us conclude that $\tilde \nu(s)$ does not touch $\partial J \cap \partial D_t$ before $\tau'$. The claim follows by taking the union over  $\epsilon>0$.
\end{proof}
	
\begin{claim}
Almost surely for all $z \notin B \cup C$, we have $|h_{B \cup C}(z)| \leq2\lambda(M+1)$.
\end{claim}

\begin{proof}
Let $(z_n:n\in \N)$ be a dense subset of $D$, such that for all $n\in \N$ the event $z_n\notin B$ has positive probability. It suffices to show that for all $n\in \N$ a.s.  $|h_{B \cup C}(z_n)| \leq2\lambda(M+1)$. On the event $z_n\notin B$ we have the following possibilities:

If $E_{t,z_n}$ does not occur for any rational $t > 0$, then $O_B(z_n)$ is contained in $O_{C}(z_n)$. Thus, $|h_{B \cup C}(z_n)| = |h_B(z_n)| \leq 2\lambda(M+1)$.

If $E_{t,z_n}$ occurs for some rational $t > 0$, from Claim \ref{loopsclose} we deduce that either $O_{C}(z_n) \subset O_B(z_n)$ or the $\CLE^{M+1}_4$ loop surrounding $z$ separates some components of $B$ from the others, see Figure \ref{contain}. Let us call $L(z_n)$ the connected component of the complement of $B \cup C$ that has this loop as part of its outer boundary. 
Importantly (see Lemma 3.11 of \cite{SchSh2}), the boundary conditions of $h_{B \cup C}$ in $L(z_n)$ are given by those of $h_B$ or $h_{C}$ everywhere but at (at most) two prime ends corresponding to the beginning and the end of the relevant $\CLE^{M+1}_4$ loop. In this case the claim follows from Lemma \ref{finite points} applied to the complement of $L(z_n)$ in $D$ (note that because the CLE$_4^{M+1}$ loop is at positive distance of $\partial D$ and because of the BTLS condition for $B$, the complement of $L(z_n)$ can have only finitely many connected components).
\end{proof}
\begin{figure}[ht!]
	\centering		
	\includegraphics[scale=0.71]{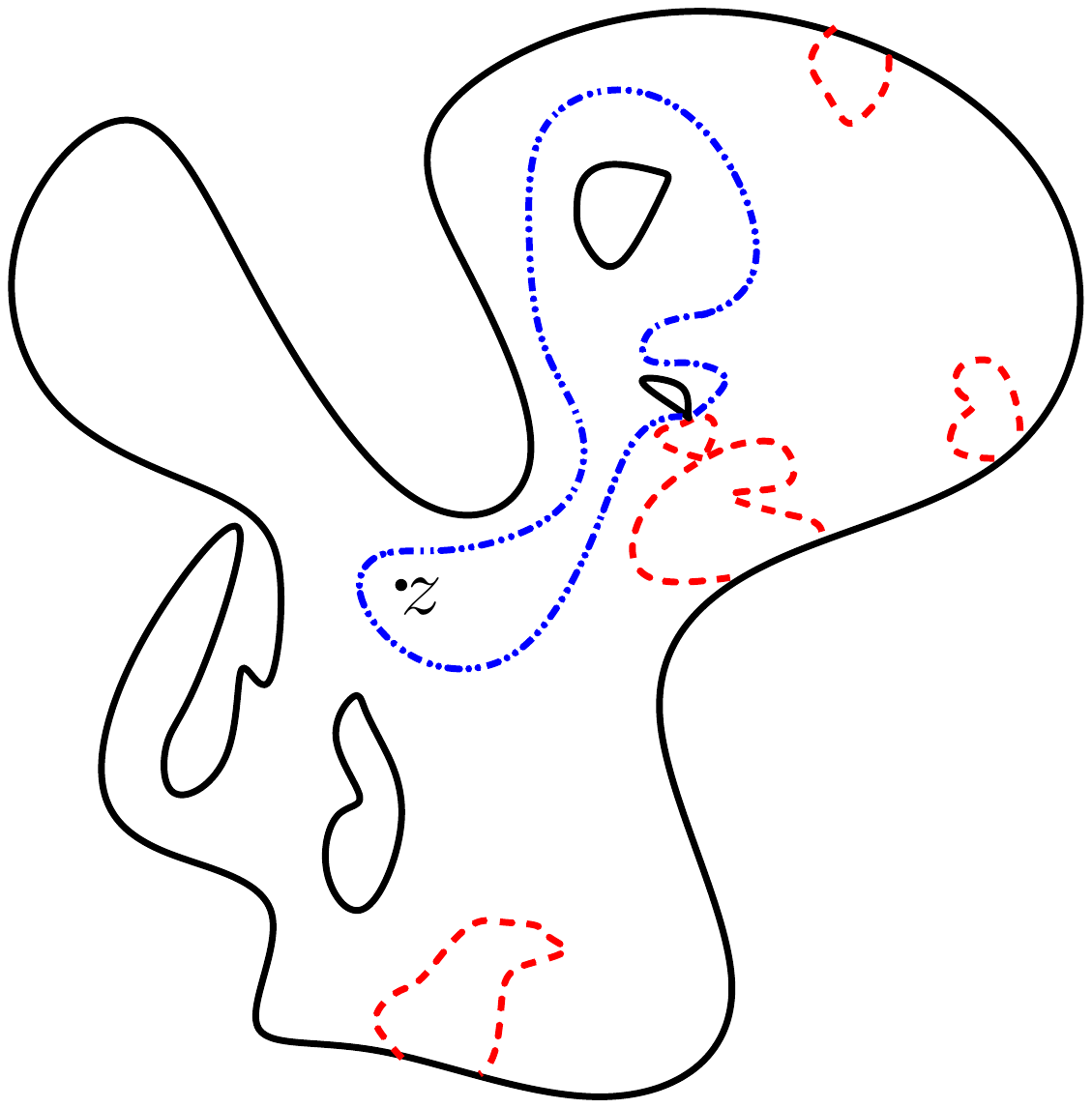}
	\caption{The continuous line represents $\partial O_B$, the dash-dotted line represents the recently closed loop and the dashed line represents the borders of the past loops traced by SLE$_4^r$.}
	\label{contain}
\end{figure}

Proposition \ref{maxbtls} now follows from Lemma~\ref{mes} applied to $A=C$ (noting that the set $C$ satisfies 
Condition ($*$)).

\subsection{Proof of Proposition \ref{cledesc}}

The proof of Proposition \ref {cledesc} goes along the same lines as that of Proposition \ref{maxbtls}. The difference lies in the fact that this time, $C$ is a CLE$_4^M$ and that 
$h_A \in \{ - 2 \lambda M, 2 \lambda M \}$.  Thus, with the same notations as before, the boundary conditions on $\partial D_t (z) \setminus K_t^z$ are in  $\{ - 2M \lambda, 2M \lambda \}$.
As above, we conclude from Lemma \ref{notouch} that the part of the radial SLE$_4 (-2)$ drawing $2 \lambda (M-1)$ versus $2 \lambda M$ level-line loop cannot exit $O_B(z)$ before 
completing that loop (one has to modify Figure \ref {fig:test}, so that the dash-dotted interface $\eta$ is now a $2M \lambda$ vs.  $2(M-1)\lambda$ interface, and the continuous boundary data is $\pm 2M \lambda$). 

Hence, it follows using the same argument as before that $|h_{A \cup C}| \le 2 \lambda M$, and then using Lemma~\ref {mes} that $A \subseteq C$ almost surely.  
But this means that for all $z$,  $|h_{A \cup C} (z)|  = 2 \lambda M = | h_A (z) |$ almost surely. Again, using Lemma \ref {mes}, we see that $C \subseteq A$ almost surely.  

\subsection {Proof of Proposition \ref {propnodisconnection}}

Suppose now that $A$ is a $K$-BTLS such that with positive probability, there exists a connected component of $A$ that is disconnected from $\partial D$. We choose some $M$ such that $2 \lambda ( M -1 )  \ge K$. We have just shown that 
almost surely, $A \subseteq C$ where $C$ is the CLE$_4^M$ coupled with the GFF.

On the other hand, with positive probability, $D \setminus A$ contains an annular open region that disconnects this connected component of $A$ from $\partial D$. Hence, there exists a deterministic such annular region $O$  
such that with positive probability, $O$ is in $D \setminus A$ and disconnects a connected component of $A$ from $\partial D$ (we call this event $E= E_O$).  

Given $A$ and $h_A$, the conditional distribution of $\Gamma^A$ is a GFF. It follows from Lemma \ref{abscnty} that on the event $E_O$, the conditional distribution of $\Gamma^A$ restricted to $O$ 
is mutually absolutely 
continuous with respect to the conditional distribution of $\Gamma$ itself restricted to $O$. It is possible to show, using Corollary \ref{change boundary2} and the fact that with positive probability the radial SLE$_4(-2)$ makes a loop inside $O$,  that with positive probability $C$ does not intersect the interior part of the complement of $O$. But this contradicts the fact that $A \subseteq C$.

{\section{BTLS with two prescribed boundary values} \label{S6}

In this section, we  describe the class of BTLS such that the harmonic function can only take two prescribed values and we determine for which values such a set  does exist.

Some aspects of the following discussion are strongly related to the $\kappa=4$ case of boundary conformal loop ensembles (and their nesting) as introduced and studied 
in \cite {MSW} (that was written up in parallel to the present paper) for general $\kappa$.

\subsection{A first special example}
Let us first describe in some detail one specific example. 
Consider a (zero boundary) GFF in the unit disc and fix two boundary points, say $-i$ and $i$. 
Consider the level line of this GFF {(i.e. for all $t\geq 0$, the curve that satisfies condition ($**$) with $F_0=0$) }.  This is an SLE$_4(-1;-1)$ from $-i$ to $i$ that is coupled with the GFF as a BTLS \cite{MS1}. It is known that this is a simple (boundary touching) continuous path $\eta$ from $-i$ to $i$ in the closed disc, and that its Minkowski dimension is almost surely equal to $3/2$. It is measurable with respect to the GFF \cite{MS1} and thus we often say that we explore the GFF to find $\eta$.

The harmonic function $h_\eta$ associated to this level line can be described as follows: First notice that the complement of the curve $\D\setminus \eta$ is a union of countably many connected  components $(D^1_j)_{j \in J}$. Any component lies either to the right or to the left of the level line (if one views the level line as going from $-i$ to $i$). In each component $D_j^1$, 
the harmonic function has boundary conditions $0$ on $\partial D_j^1 \cap \partial \D$. On $\partial D_j^1 \cap \eta$ the boundary condition is either $\lambda$ or $-\lambda$, depending on whether $D_j^1$ is on the left or on the right of $\eta$. 

As inside each component $D_j^1$ there is an independent GFF with these boundary conditions, we can iterate: Suppose for example that the $D_j^1$ lies to the right of $\eta$ so that $\partial D_j^1$ is divided into two arcs, one of which is an excursion of 
$\eta$ away from the $\partial D$, and the other one is a counter-clockwise arc from $x_j$ to $y_j$ of $\partial D$. We now explore the level line of this GFF  from $x_j$ to $y_j$ with this boundary data (i.e. for all $t\geq 0$, the curve that satisfies condition ($**$) with $F_0$ equal to the given boundary data). This level line has the law of an SLE$_4 ( -1)$ process from $x_j$ to $y_j$ and is again a simple curve. 
We proceed in a symmetric way in the connected components that lie above $\eta$. 
In this way, we obtain a new BTLS $A_1$, for which the harmonic function $h_{A_1}$ is defined via the boundary conditions indicated in Figure \ref{fig:firstiteration}.

The iteration then further proceeds by exploring additional level lines (which are usual SLE$_4 (-1)$ processes with just one marked point) in each of the remaining  connected components which have a part of $\partial D$ on their boundary. One then defines a second layer of loops and one proceeds iteratively. 
We then consider the closure $A_{-\lambda, \lambda}$ of the union of all the traced level lines.

Clearly, after any finite number of iterations in the previous construction, one has a $\lambda$-BTLS, and therefore a subset of the CLE$_4^2$ by our previous results. 
Hence, $A$ is itself a subset of the CLE$_4^2$ and therefore a BTLS. It is also easy to see that a given point $z \in D$ is almost surely contained in a loop cut out after finitely many iterations, so that the harmonic function associated to $A$ takes its values in $\{ - \lambda, \lambda \}$. 

\begin{figure}[ht!]	   
	\centering
		\includegraphics[scale=0.62]{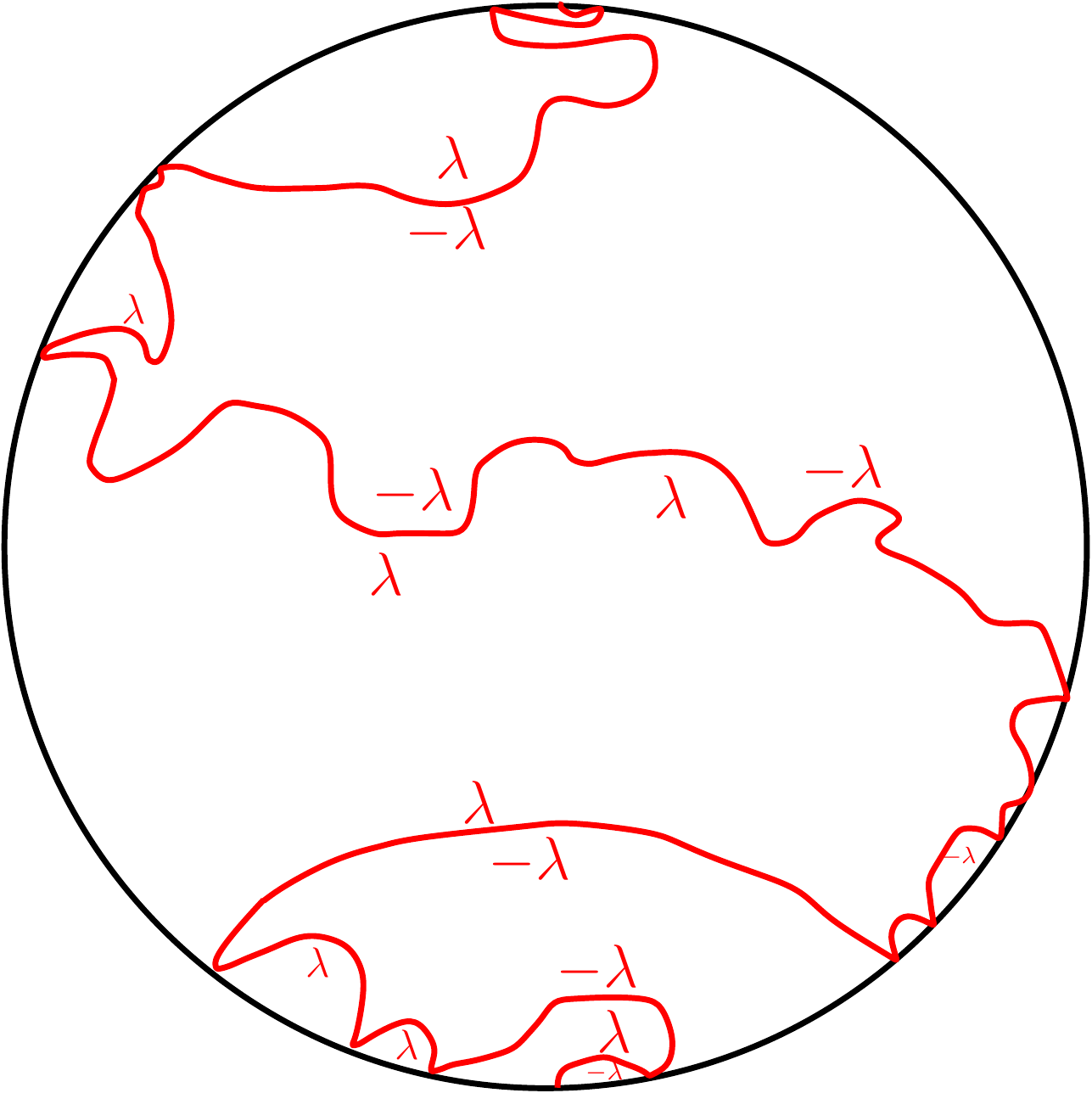}
		\includegraphics[scale=0.62]{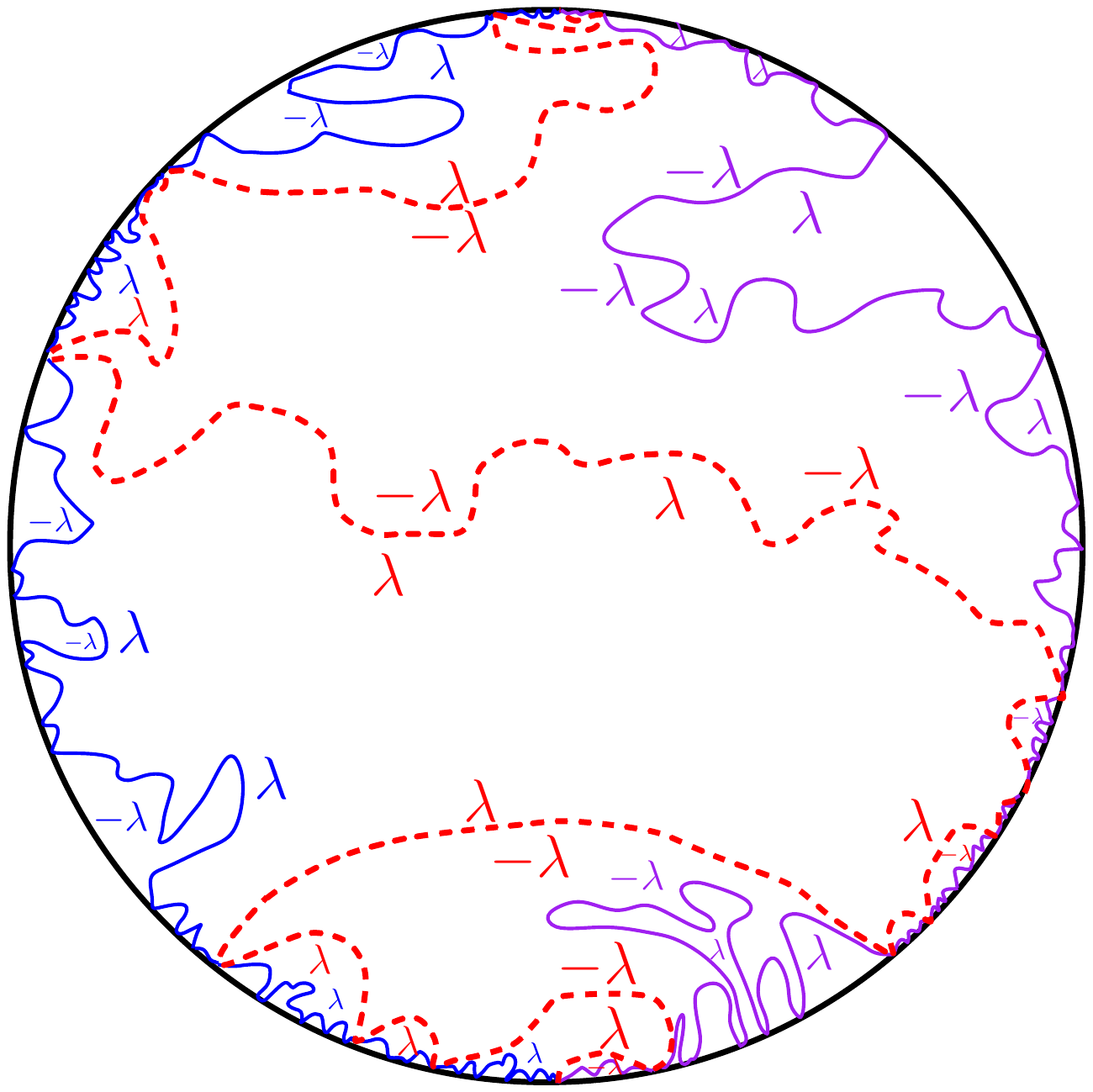}
	\caption {First iterations: The SLE$_4(-1;-1)$ and its boundary conditions on the left. The SLE$_4(-1)$ in the second layer on the right picture that creates loops with $\pm \lambda$ boundary conditions.}
	\label {fig:firstiteration}
\end{figure}

We can make the following observations about this set $A_{-\lambda, \lambda}$: 
\begin {itemize}
 \item As opposed to the CLE$_4$, the set $A_{-\lambda, \lambda}$ is just made out of the union of all SLE$_4$-type paths. For instance, each excursion of $\eta$ away from $\partial \D$
 is on the boundary between two connected components of $A_{-\lambda, \lambda}$ (one $+ \lambda$ loop to its right and one $-\lambda$ loop to its left). This indicates that the Hausdorff dimension of $A_{-\lambda, \lambda}$ is almost surely $3/2$ (we will come back to this later). 
 \item Because we have only used measurable sets in the construction $A_{-\lambda,\lambda}$, we know it is measurable function of the underlying GFF.
 \item  In addition, it comes out that the set $A_{-\lambda,\lambda}$ is the only BTLS with boundary values in $\{-\lambda,\lambda \}$. In particular its law does not depend on the arbitrary choices of the start and end points in the previous layered construction. We prove it below in a more general context.
\end {itemize}

\begin {rem}[A new construction of CLE$_4$] \label{newconsCLE4}Note that, similarly to the construction of CLE$_4^M$, we can iterate the construction of the set $A_{-\lambda, \lambda}$ to construct a BTLS with harmonic function that takes values in $\{ -2 \lambda, 2 \lambda \}$. Indeed, for each $z$ belonging to a given  countable dense  subset of $D$, 
one can iterate the construction in the component containing $z$ until the boundary values are in $\{-2\lambda, 2\lambda\}$. At each step in the construction one has a $2\lambda-$BTLS which 
is contained in CLE$_4^2$ and therefore, in the 
limit also, one still has a BTLS that is contained in the CLE$_4^2$. But from Proposition \ref{cledesc} it now follows that the obtained set is exactly the CLE$_4$. This therefore provides an alternative construction of CLE$_4$ (and of its iterated nested versions CLE$_{4,m}$ and CLE$_4^M$) that builds only on the coupling of the chordal SLE$_4 (-1;-1)$ and SLE$_4(-1)$ process with the GFF. Notice that in this case the measurability of the CLE$_4$ just follows from that of the respective SLE$_4(-1; -1)$ and SLE$_4(-1)$ processes.

As we point out in Section \ref{Sdim}, there is also a direct way to see that 
this set $A_{-\lambda, \lambda}$ and its iterates are thin (without using the relation to CLE$_4^M$ as we just did). Hence, we indeed obtain a stand-alone construction of CLE$_4$ and derivation of its properties. Interestingly, we do not know how to show that this construction gives the same law as CLE$_4$ without using the coupling with the GFF.
\end {rem}

\subsection {General sets $A_{-a,b}$ and proof of Proposition \ref {cledesc2}} 
We first construct a BTLS $A_{-a, b} $ such that $h_A$ takes its values in $\{ -a, b \}$, for all given pairs $(a,b)$ such that  $-a \leq 0 \leq  b$ 
and $b+a \ge 2 \lambda$. This generalizes our previous constructions of $A_{-\lambda,  \lambda}$ and of CLE$_4^M$ (that will be our $A_{-2M \lambda, 2M\lambda}$). 
We then prove their uniqueness, the monotonicity of $A_{-a,b}$ with respect to $a$ and $b$,  and we show that there exist no BTLS with $h_A \in \{ -a, b\}$ when
$b+a < 2\lambda$ (unless $a$ or $b$ are equal to $0$, in which case one can take the empty set).

\subsubsection*{Construction of $A_{-a,b}$ and measurability} 	We first the construct  $A_{-a,b}$ for some ranges of values of $a$ and $b$, and then describe the general case: 
	
	\begin {itemize}
\item \underline{\textit{$a=0$ or $b=0$:}} We set $A_{-a,b}=\emptyset$ and the corresponding harmonic function takes the value $0$ everywhere.

\item\underline{\textit{$a=-n_1\lambda$ and $b=n_2\lambda$, where $n_1$ and $n_2$ are positive integers:}} Note that similarly to the construction of CLE$_4$ in \ref{newconsCLE4}, we can iterate the construction of the set $A_{-\lambda, \lambda}$. Indeed, pick a countable number of dense $z \in D$ and iterate the construction in the component containing $z$ until the boundary values are in $\{-n_1\lambda, n_2\lambda\}$.

\item 
\underline{\textit{$a+b = 2 \lambda$:}} Set $c := (b-a)/2 \in (- \lambda, \lambda )$ and 
 repeat exactly the same construction as above, except that one now traces $c$-level lines i.e. $c-\lambda $ vs $c+ \lambda $ interfaces iteratively instead of $-\lambda$ vs. $\lambda$ interfaces. Exactly the same construction 
and the same arguments lead to the construction of a BTLS $A_{c - \lambda, c + \lambda}$ such that the corresponding harmonic function takes its values in $\{c- \lambda, c + \lambda \}$. These sets are called the boundary conformal loop ensembles for $\kappa=4$ in \cite {MSW}. Notice that in these sets each interior boundary arc is shared by two components of the complement.

\item \underline{\textit{$a+b = n\lambda$ where $n \geq 3$ is an integer}}: Define $c\in (-\lambda, \lambda)$ such that there exists two non-negative integers $n_1, n_2$ with $a=c-n_1\lambda$ and $b =c+n_2\lambda$. Starting from $A_{c - \lambda, c + \lambda}$ we now iterate copies of $A_{-(n_1 -1) \lambda,  n_2 \lambda}$ (resp. of $A_{-n_1\lambda,  (n_2 -1) \lambda}$) in the connected components of the complement of $A_{c -  \lambda, c +\lambda}$  depending on the value of the harmonic function.  
 
\item
\underline{\textit{General case with $b+a > 2\lambda$}}: by symmetry, we can suppose that $a > \lambda$. Define $c \in [0, \lambda)$ and $n_1, n_2 \ge 0$ such that $b = c + n_1 \lambda$ and $b - n_2 \lambda \in [-a  ,  -a + \lambda)$.
Note that $b - n_2 \lambda < 0 $ and that $n_2 \ge 2$. Now consider an $A_{b-n_2 \lambda,b}$. In the connected components where the harmonic function is $b$, we stop. In those components where 
the harmonic function is $b - n_2 \lambda$, we iterate copies of $A_{-d, -d + n_2 \lambda}$ where $d =  (a+ b) - n_2 \lambda \in [0, \lambda)$. 
In this way, we obtain a BTLS whose harmonic function takes values in $\{-a,b-d,b\}$. We stop in all components where the harmonic value is either $-a$ or $b$, and iterate copies of $A_{ d - n_2 \lambda, d }$ in the other components. The harmonic function of the resulting BTLS takes values in $\{-a,b-n_2\lambda,b\}$. We continue this way, iterating $A_{-d, -d + n_2 \lambda}$ or $A_{ d - n_2 \lambda, d }$, in the components labelled $b-n_2\lambda$ or $b-d$, respectively. 
At each step of the iteration, we have a BTLS that is contained in CLE$_4^M$ with, say, $M = n_1 + n_2 +1$. The closure of the union of all the  constructed sets is then the  desired set $A_{-a,b}$.

\end {itemize}~

We make the following observations about the $A_{-a,b}$ constructed above:
	\begin{enumerate}[(i)]
\item In the construction we only need to use level lines whose boundary values are in $[-a,b]$.
\item For a fixed point $z\in D$ a.s. we only need a finite number of level lines to construct the loop of $A_{-a,b}$ surrounding $z$.
\item From the measurability of the level lines used in the construction, it follows that the sets $A_{-a,b}$ are measurable with respect to the underlying GFF. 
\end{enumerate}

\subsubsection*{Uniqueness}  To show uniqueness, we follow   loosely the strategy of the proof of Proposition \ref{cledesc}. 

Suppose that $\tilde A$ is another BTLS coupled with the same GFF, such that $h_{\tilde A}$ takes its values in $\{ - a,b \}$ and such that conditionally on $\Gamma$, $A_{-a,b}$ and $\tilde A$ are independent. Consider some $z \notin \tilde A$ and denote by $O(z)$ the component of $z$ in $D \setminus \tilde A$.  Now, we claim that almost surely no level line in the construction of the component of $z$ in $D \setminus A_{-a,b}$ can make an excursion inside of $O(z)$: Indeed, suppose that with positive probability a level lines does an excursion inside $O(z)$. On this event, using (ii) we can consider the first level line entering $O(z)$. Then, on the one hand this level line cannot exit $O(z)$ through the boundary of $\tilde A$, due to  Lemma \ref{notouch} and (i). On the other hand, it is also almost surely a simple path. Thus, it cannot exit $O(z)$ at all and we obtain a contradiction. As this holds for a countable dense family of $z$, we obtain that $A_{-a,b}\subseteq \tilde A$. We conclude using Lemma~\
ref{mes} with $k=(-a+b)/2$.

In particular, this implies that the arbitrary choices of points in the construction of $A_{-\lambda,\lambda}$ and also in the constructions of $A_{-a,b}$ do not matter.

\subsubsection*{Monotonicity}

Suppose $[-a,b] \subset [-a', b']$ and $-a < 0 < b$ with $b+a > 2\lambda$. Start with $A_{-a, b}$, and then
explore $A_{a-a',a+b'}$ in all the connected components of its complement where the boundary values are $a$  and $A_{-a'-b, b'-b}$ in the others. We obtain a BTLS with boundary values in $\{-a', b'\}$. By uniqueness it follows that the obtained set is indeed equal to $A_{-a',b'}$ and by construction it contains $A_{-a,b}$.
 
\subsubsection*{There are no BTLS $A$ with $h_A \in \{ -a,b \}$ when $a$ and $b$ are non-zero and $a+b < 2\lambda$}
 First, one can discard the case where $-a$ and $b$ have the same sign because the mean value of the field has to remain $0$. When $-a < 0 < b$ and $b + a < 2 \lambda$,
suppose that $A$ is a BTLS with $h_A \in \{ -a , b \}$. Exploring $A_{-a-b, -a-b + 2 \lambda}$ in those connected components of the complement of $A$ 
 where the harmonic function $h_A = b$, we see that  $A \subset A_{-a, -a + 2 \lambda}$. In particular, 
 note that a connected component of $D \setminus A$ where the corresponding harmonic function $h_{A}$ is equal to $-a$ remains a connected component of
 $D \setminus A_{-a, -a + 2 \lambda}$. Such a connected component has boundary value $-a$, but has no boundary arc that is shared with a component of $D \setminus A_{-a, -a + 2 \lambda}$ where 
 the harmonic function is $-a+ 2 \lambda$ on the other side. 
 This leads to a contradiction with an observation made above, because we know from the construction of $A_{-a, -a  + 2 \lambda}$ that all interior boundary arcs are shared by two components of the complement of this set.

 \subsection{Some additional properties of the sets $A_{-a,b}$}

\subsubsection*{Labels of $A_{-a,b}$} 
\label {commenting}

Conformal invariance of the GFF implies the conformal invariance of $A_{-a,b}$ and thus we see that $\P ( h_{A_{-a,b}} (z) = -a )$ does not depend on $z \in D$. Using the fact that $\E ( (\Gamma,1) ) = 0$, we therefore see that for all $z \in D$, 
$$ \P ( h_{A_{-a,b}} ( z)  = - a ) = \frac {b}{a + b}\   \hbox { and }\     \P ( h_{A_{-a,b}} ( z)  = b ) = \frac {a}{a + b},$$
as one might have expected.
  
As mentioned in Section \ref{CLE4}, by combining the radial SLE$_4(-2)$ construction and the reversibility of SLE$_4$ one can show that it is possible to first sample the family $(O_j)$, and then the labels $\eps_j$ using independent fair coins. The idea is that one can define the symmetric radial SLE$_4 (-2)$ using a Poisson point process of SLE$_4$ bubbles, which is invariant under resampling of the orientations of the bubbles, see e.g. \cite {WW}. 

This conditional independence between the heights in different domains feature is 
specific to CLE$_4$.  For instance, for CLE$_4^M$ the existence of a correlation between 
the heights in different domains is clear from the construction. For $A_{-\lambda, \lambda}$, 
the situation is actually quite reversed: conditioned on $A_{-\lambda, \lambda}$, one  single fair coin toss that decides the sign of the harmonic function at the origin is enough to 
determine the harmonic function in all the other connected components of the complement of $A_{-\lambda, \lambda}$. This follows from the fact that any two neighbouring components have different heights.
In the case of $A_{c - \lambda, c + \lambda}$ for $c \not= 0$, one does not even need to toss a coin, as the asymmetry (and conformal invariance) makes it in fact possible to detect almost surely the sign of the harmonic function at the origin. 
From this point of view, it is  even quite intriguing that by iterating $A_{-\lambda, \lambda}$, one obtains a CLE$_4$ where the signs of the heights are independent in the different components. In fact it is possible to determine precisely to which extent the function $h_A$ is a measurable functions of $A$ when $A_{-a,b}$ --  this will be a topic of a follow-up note.
 
 \subsubsection* {Dimension of $A_{-a,b}$}\label{Sdim}
 
 The goal of the present paragraph is to derive the following fact:
 \begin {prop} For each given $z$, 
the random variable  $\log (\text{\emph{crad}} (z, \D )) - \log (\text{\emph{crad}} (z, \D  \setminus A_{-a, b}))$ is distributed like a constant times the exit time from $(- a \pi / (2 \lambda), b \pi / (2 \lambda))$ by a one-dimensional Brownian motion $B$ started from $0$. 

In particular, for each given $z$, the probability that $d (z, A_{-a ,b}) < r$ is (up to constants) comparable to $r^{s}$ for $s =  s_{-a,b}:= 2\lambda^2 /(b+a)^2$ as $r \to 0$. 
 \end {prop}

\begin{proof}
 Let us first focus on the case of the set $A = A_{- \lambda, \lambda}$.  Fix a point $z \in \D$ and note that the construction of the set $A$ is 
obtained via a continuously increasing family of sets $(A_t, t \ge 0)$ that correspond to the concatenation of the various chordal SLE$_4 (-1;1)$ and SLE$_4 ( -1)$ processes that one iterates. All these sets $A_t$ are clearly local sets, and the value of 
the corresponding harmonic function $h_t (z)$ at $z$ is always in $[- \lambda, \lambda]$ (as its boundary values are in $\{ 0, \lambda, - \lambda \}$). Furthermore, by definition it is a local martingale, and therefore a martingale. 
We know that it converges to either $+ \lambda$ or $- \lambda$ as $t \to \infty$. 

On the other hand, we know (see e.g.. \cite {MS1}) that when $R_t$ is a continuously increasing family of local sets then $h_{R_t} (z)$ evolves like a $2\lambda/\pi$ times the standard Brownian motion when parametrized by the decrease of the log-conformal radius 
of $\D \setminus R_t$ seen from $z$. This therefore implies that $\log (\crad (z, \D))  - \log (\crad (z, \D  \setminus A))$ is distributed like the exit time of $(- \pi/2, \pi/2)$ by a one-dimensional Brownian motion. The tail estimate then follows from \cite{SchShW}. 

Exactly the same argument can be applied to all $A_{-a, b}$'s that we have constructed -- one just needs to note that in our iterative procedure, all the iterations are independent and the harmonic functions $h_t$ always remain in $[-a , b]$.
\end{proof}

Note that this argument can be used to see that $A_{-\lambda,\lambda}$ is indeed a BTLS with upper Minkowski dimension almost surely not larger than $3/2$ and that CLE$_4$ obtained by iterations of $A_{-\lambda,\lambda}$ is indeed a BTLS (as it satisfies $(3)$). 
}

\medbreak

 \subsection* {Acknowledgements.} This work was supported by the SNF grant \#155922. The authors are part of the NCCR Swissmap. 
 The authors would like to thank the Isaac Newton Institute for Mathematical Sciences, Cambridge, as well as the Clay foundation, for  hospitality and support during the program \emph{Random Geometry} where a part of this work was undertaken. 
They also thank the anonymous referees for their careful reading and their comments.

\end{document}